\RequirePackage{fix-cm}
\RequirePackage{amsmath}
\documentclass[smallextended]{svjour3hack}       
\smartqed  
\usepackage{graphicx}
\usepackage{amssymb}
\usepackage{placeins}

\usepackage[hypertexnames=false,colorlinks=true,breaklinks=true,bookmarks=true,urlcolor=blue,citecolor=blue,linkcolor=blue,bookmarksopen=false,draft=false]{hyperref}

\usepackage[linesnumbered, boxed, english, onelanguage, vlined]{algorithm2e}

\newtheorem{jremark}{Remark}

\makeatletter
\let\c@proposition\c@theorem
\let\c@corollary\c@theorem
\let\c@lemma\c@theorem
\let\c@definition\c@theorem
\let\c@example\c@theorem
\let\c@remark\c@theorem
\let\c@jremark\c@theorem
\makeatother

\makeatletter
\renewcommand*{\thetable}{\arabic{table}}
\renewcommand*{\thefigure}{\arabic{figure}}
\let\c@table\c@figure
\makeatother

\vbadness=\maxdimen

\makeatletter
\renewcommand*{\top}{%
  {\mathpalette\@transpose{}}%
}
\newcommand*{\@transpose}[2]{%
  \raisebox{\depth}{$\m@th#1\mathsf{T}$}%
}
\makeatother

\DeclareMathOperator{\rank}{rank}
 \journalname{Journal of Global Optimization}

\begin{document}

\title{Experimental analysis of local searches for sparse
reflexive generalized inverses}
\author{Marcia Fampa \and Jon Lee \and \\ Gabriel Ponte \and Luze Xu}%
\authorrunning{Fampa, Lee, Ponte \& Xu} 

\institute{M. Fampa and G. Ponte\at
            Federal University of Rio de Janeiro \\
              \email{fampa@cos.ufrj.br}, \email{gabrielponte@poli.ufrj.br}            
           \and
           J. Lee and L. Xu \at
            University of Michigan \\
              \email{jonxlee@umich.edu},     \email{xuluze@umich.edu}
}

%

\date{Received: date / Accepted: date}

\maketitle
\begin{abstract}
The well-known M-P (Moore-Penrose) pseudoinverse is used in several linear-algebra applications; for example, to compute least-squares solutions of inconsistent  systems of linear equations. Irrespective of whether a given matrix is sparse, its M-P pseudoinverse can be completely dense, potentially leading to high computational burden and numerical difficulties, especially when we are dealing with high-dimensional matrices. The  M-P pseudoinverse is uniquely characterized by four properties, but not all of them need to be satisfied for some applications. In this context, Fampa and Lee (Oper. Res. Letters,  46:605--610, 2018) and Xu, Fampa, Lee and Ponte (SIAM J. on Optimization, to appear)  propose local-search procedures to construct sparse block-structured generalized inverses that satisfy only some of the M-P  properties. (Vector) 1-norm minimization is used to induce sparsity and to keep the magnitude of the entries under control, and  theoretical results limit the distance between the 1-norm of the solution of the local searches and the minimum 1-norm of generalized inverses with corresponding properties.
We have implemented several local-search procedures based on results presented in these two papers and make here an experimental analysis of them, considering their application to randomly generated matrices of varied dimensions, ranks, and densities. Further, we carried out a case study on a real-world data set.
\end{abstract}

\section{Introduction}

The well-known M-P  (Moore-Penrose) pseudoinverse, independently discovered  by  E.H. Moore and R. Penrose, is used in several linear-algebra applications --- for example, to compute least squares solutions of inconsistent  systems of linear equations. If $A=U\Sigma V^\top$ is the real singular value decomposition of $A$ (see \cite{GVL1996}, for example),
then the
M-P pseudoinverse of $A$ can be defined as $A^{\dagger}:=V\Sigma^{\dagger} U^\top$, where $\Sigma^{\dagger}$
has the shape of the transpose of the diagonal matrix $\Sigma$, and is derived from $\Sigma$
by taking reciprocals of the non-zero (diagonal) elements of $\Sigma$ (i.e., the non-zero
singular values of $A$).
The following theorem gives a fundamental characterization of the M-P pseudoinverse.
\begin{theorem}[\cite{Penrose}]
For $A\in\mathbb{R}^{m \times n}$, the M-P pseudoinverse $A^{\dagger}$ is the unique 
 $H\in\mathbb{R}^{n \times m}$ satisfying:
	\begin{align}
		& AHA = A \label{property1} \tag{P1}\\
		& HAH = H \label{property2} \tag{P2}\\
		& (AH)^{\top} = AH \label{property3} \tag{P3}\\
		& (HA)^{\top} = HA \label{property4} \tag{P4}
	\end{align}
\end{theorem}

Following \cite{RohdeThesis}, we say that a \emph{generalized inverse} is any $H$ satisfying  \ref{property1}. The property \ref{property1} is particularly important in our context; without it,
the all-zero matrix --- extremely sparse and carrying no information at all about $A$ --- would
satisfy the other three properties.

A generalized inverse is \emph{reflexive} if it satisfies \ref{property2}.  Theorem 3.14 in \cite{RohdeThesis} states that: ($i$) if $H$ is a generalized inverse of $A$, then $\mathrm{rank}(H)\ge\mathrm{rank}(A)$, and ($ii$) a generalized inverse $H$ of $A$ is reflexive if and only if $\mathrm{rank}(H)=\mathrm{rank}(A)$. Therefore, enforcing \ref{property2}
gives us the lowest possible rank of a generalized inverse --- a very desirable property.

Finally, following \cite{XuFampaLee},  we say that $H$ is \emph{ah-symmetric} if it satisfies \ref{property3}. That is, ah-symmetric means that $AH$ is symmetric.
If $H$ is an ah-symmetric generalized inverse, then $\hat{x}:=Hb$ solves $\min\{\|Ax-b\|_2:~x\in\mathbb{R}^n\}$ (see \cite{FFL2016,campbell2009generalized}).
So  not  all of the M-P properties are required for a generalized inverse to solve a key problem.

Even if a given matrix is sparse, its M-P pseudoinverse can be completely dense, often leading to a high computational burden in its applications, especially when we are dealing with high-dimensional matrices.
 Therefore, to  avoid computations with  high-dimensional dense matrices,  it is interesting to consider the construction of sparse generalized inverses that satisfy only a proper  subset of $\{$\ref{property2}, \ref{property3}, \ref{property4}$\}$.   In this context, \cite{FampaLee2018ORL} and \cite{XuFampaLee} propose local-search procedures to construct reflexive generalized inverses,  ah-symmetric reflexive generalized inverses, and  in case $A$ is symmetric,  symmetric reflexive generalized inverses. The purpose of the procedures is the construction of sparser matrices than the M-P pseudoinverse, without losing some of its important properties. In   \cite{FampaLee2018ORL,XuFampaLee},  (vector) 1-norm minimization is used to induce sparsity (leading to less computational burden in applications) and to keep the magnitude of the entries under control (leading to better numerical stability in applications).  Therefore, at each iteration of the local-search procedures, the overall goal is to decrease the 1-norm of the constructed matrix $H$.

The generalized inverses constructed by the procedures have the following very nice features: they have block structure, i.e.,  they have all non-zero entries confined to a selected choice of columns (and, sometimes,  also of rows), they are reflexive,  they have a bounded number of non-zero entries, and they  have 1-norm within a provable factor of the minimum 1-norm of generalized inverses with corresponding properties.

Our goal in this paper is to develop and analyze through numerical experiments,  the performance of local-search procedures based on the ideas  presented  in \cite{FampaLee2018ORL,XuFampaLee}, and to see how tight are the bounds presented for the 1-norms of the constructed matrices $H$, considering  randomly generated input matrices $A$ with varied dimensions, ranks, and densities.  We have implemented different local-search procedures for each  case studied, more specifically, the cases where we construct ($i$) a reflexive generalized inverse,  ($ii$) an ah-symmetric  reflexive generalized inverse, and ($iii$) a symmetric reflexive generalized inverse.
We propose a method for constructing an initial solution for the local searches; interestingly, this
turns out to be a rather difficult numerical task at large scale, even though in theory it is
rather trivial. 
We propose and compare local searches  with updates performed with the best improvement (`BI') obtained in the neighborhood of the starting solution, and  with updates performed with the first improvement (`FI') obtained. We analyze local searches that   consider as the criterion for improvement, the increase in the absolute determinant of an  $r\times r$ non-singular submatrix of the given rank-$r$ matrix $A$, which are based on  theoretical results presented in \cite{FampaLee2018ORL,XuFampaLee}. These procedures are identified in the paper with the notation `det'.   We also propose a local search that  considers a more natural criterion for improvement, the decrease in the 1-norm of the constructed matrix $H$, and is identified with `norm'.
Observing the behavior of these local searches
leads us to combine the `det' with the  `norm' searches. Aiming at reaching matrices with smaller norms, we apply  hybrid procedures that perform local searches based on the decrease of the 1-norm of $H$, starting from the output of a local search based on the increase of the absolute determinant of the submatrix of $A$.

The algorithms proposed were coded in Matlab R2019b. 
To evaluate the solutions obtained by them, we  solve the linear programming (LP) problems described in the next sections, with Gurobi v.9.0.2. We ran the experiments on a
16-core machine (running Windows Server 2016 Standard):
two Intel Xeon CPU E5-2667 v4 processors
running at 3.20GHz, with 8 cores each, and 128 GB of memory.

In \S\ref{sec:ginv}, we present our results for generalized inverses.
In \S\ref{sec:ahsymginv},  we present our results for ah-symmetric  generalized inverses.
In \S\ref{sec:symginv}, we present our results for symmetric generalized inverses (applied to symmetric input matrices). In \S\ref{seccase}, we present a case study where we apply our algorithm for ah-symmetric generalized inverses to real data.
In \S\ref{sec:conc}, we make some brief concluding remarks.

Before continuing, we wish to mention that an earlier approach to
constructing sparse generalized inverses was developed in \cite{FFL2019}.
Unfortunately those methods, based on solving convex relaxations (LP and convex QP),
scale very poorly.
The failure of those methods to scale efficiently
led to the investigations in \cite{FampaLee2018ORL} and \cite{XuFampaLee},
which in turn motivated our present work. \cite{dokmanic1,dokmanic2,dokmanic}
presents an additional prior approach, based also on LP,
for constructing sparse left and right pseudoinverses.

In what follows, for succinctness, we use vector-norm notation on matrices: we write $\|H\|_1$ to mean $\|\mathrm{vec}(H)\|_1$, and $\|H\|_{\max}$ to mean $\|\mathrm{vec}(H)\|_{\max}$ (in both cases, these are not the usual induced/operator matrix norms). We use $I$ for an identity matrix and $J$ for an all-ones matrix. Matrix dot product is indicated by $\langle X, Y\rangle=\mathrm{trace}(X^\top Y):=\sum_{ij}x_{ij}y_{ij}$.  We use $\sigma_{\min}(A)$ for the minimum singular value of  $A$. We use $A[S,T]$ for the submatrix of $A$ with row indices $S$ and column indices $T$; additionally, we use $A[S,:]$ ( resp., $A[:,T]$) for the submatrix of $A$ formed by the rows $S$ (resp., columns $T$).
Finally, if $A$ is symmetric and $S=T$, we use $A[S]$ to represent the principal submatrix of $A$ with  row/column indices $S$.

\section{Generalized inverse}\label{sec:ginv}

The local-search procedures for the reflexive generalized inverse are based on the block construction procedure proposed in \cite{FampaLee2018ORL}. More specifically they are based on Theorem \ref{rbyrsol}, Definition \ref{def2018}, and Theorem \ref{thm:approx}, presented next.

\begin{theorem} [\cite{FampaLee2018ORL}]\label{rbyrsol}
For $A\in\mathbb{R}^{m\times n}$, let $r:=\rank(A)$.
Let $\tilde{A}$ be \emph{any} $r\times r$ non-singular submatrix of $A$.
Let $H\in\mathbb{R}^{n\times m}$ be such that its submatrix that
corresponds in position to that of $\tilde{A}$ in $A$ is equal to
$\tilde{A}^{-1}$, and other positions in $H$ are zero.
Then $H$ is a reflexive generalized inverse of $A$.
\end{theorem}

\begin{definition}[\cite{FampaLee2018ORL}]\label{def2018}
For $A\in \mathbb{R}^{m\times n}$, let $r:=\rank(A)$.
For $\sigma$ an ordered subset of $r$ elements from $\{1,\ldots,m\}$ and $\tau$ an ordered subset of $r$ elements from  $\{1,\ldots,n\}$,
let $A[\sigma,\tau]$ be the $r\times r$ submatrix of $A$ with
row indices $\sigma$ and column indices $\tau$.
For fixed $\epsilon \geq 0$,
if $|\det(A[\sigma,\tau])|$ cannot be increased by a factor of more than
$1+\epsilon$  by either swapping
an element of $\sigma$ with one from its complement
or swapping
an element of $\tau$ with one from its complement, then
we say that $A[\sigma,\tau]$ is a \emph{$(1+\epsilon)$-local maximizer
for the absolute determinant} on the set of $r\times r$ non-singular submatrices of
$A$.
\end{definition}

\begin{theorem}[\cite{FampaLee2018ORL}]\label{thm:approx}
For $A\in \mathbb{R}^{m\times n}$, let $r:=\rank(A)$. Choose $\epsilon\geq 0$,
and let $\tilde{A}$ be a $(1+\epsilon)$-local maximizer
for the absolute determinant on the set of $r\times r$ non-singular submatrices of
$A$.
Construct $H$ as per Theorem \ref{rbyrsol}.
Then $H$ is a (reflexive) generalized inverse (having at most $r^2$
non-zeros), satisfying
$\| H \|_1 \leq r^2(1+\epsilon)^2 \| H_{opt} \|_1$,
where $H_{opt}$ is a 1-norm minimizing generalized inverse of $A$~.
\end{theorem}

We note that the $\epsilon$ of Definition \ref{def2018} and Theorem \ref{thm:approx} is used
in \cite{FampaLee2018ORL} to gain polynomial running time in $1/\epsilon$. For the purpose of actual computations,
our observation has been that $\epsilon$ can be chosen to be zero.
We further note that in \cite{XuFampaLee}, we demonstrated that the
bound in Theorem \ref{thm:approx} is the best possible. However, we
will see in our experiments that the bound is overly pessimistic
by a wide margin.

The idea of our algorithms is to select an $r\times r$  non-singular submatrix $\tilde{A}$ of $A$, and construct the reflexive generalized inverse with the inverse of this submatrix, as described in Theorem \ref{rbyrsol}.
The non-zero entries of $H$ will be the non-zero entries of $\tilde{A}^{-1}$. Guided by the result in Theorem \ref{thm:approx}, the `det' searches aim at selecting a submatrix $\tilde{A}$ that is a local maximizer
for the absolute determinant on the set of $r\times r$ non-singular submatrices of
the given matrix $A$. In an attempt to construct matrices $H$ with smaller 1-norm, the `norm' searches more directly try to decrease the 1-norm of the matrix constructed at each iteration.

In the following, we discuss how the  test matrices $A$ used in our computational experiments were generated, how we select the initial submatrix of $A$ to initialize the local searches, and we give details of the algorithms and present numerical results.

To analyze the local-search procedures proposed, we  compare their solutions to the solution of
a natural LP problem, identified below as $P_1$. Its optimal  solution value corresponds to  $\| H_{opt} \|_1$,  where as defined in Theorem \ref{thm:approx}, $H_{opt}$  is a 1-norm minimizing generalized inverse of $A$.

\[
\begin{array}{lll}
(P_1) \  z_{P_1}:=&\min  & \left\langle J , T \right\rangle~,  \\
& \mbox{s.t.:}    &T - H \geq 0~,\\
		     &&T + H \geq 0~,\\
		     &&AHA=A~.
\end{array}
\]

\subsection{Our test matrices}
To test the proposed local-search procedures, we randomly generated 462  matrices with varied dimensions,  ranks, and  densities, with the Matlab function \emph{sprand}.
The function generates a random  $m\times n$ dimensional matrix $A$ with approximate
density $d$ and singular values
given by the non-negative input vector $rc$.
 The number of non-zero singular values in  $rc$
is of course the desired rank $r$. The matrix is generated by sprand using
random plane rotations applied to a diagonal matrix with the given singular values.
For our experiments, we selected the $r$ nonzeros of $rc$
as the decreasing vector $M\times(\rho^{1},\rho^{2},\ldots,\rho^{r})$, where $M=2$, and $\rho=(1/M)^{(2/(r+1))}$.
The shape of this distribution is concave (as is the case for many matrices that one encounters),
and moreover, the entries are not extreme (always between $1/2$ and $2$), and the
product is unity, so we can reasonably hope that the numerics may not be terrible.

We divide our instances into the following three categories:
\begin{itemize}
\item Small: 90 instances. 5 with  each of the 18  combinations of the following parameters:  $m=n=50,80,100$; $r=0.1\times n,0.5\times n$;  $d = 0.25, 0.50, 1.00$.
\item Medium: 360 instances. 30 with  each of the 12 combinations of the following parameters:  $m=n=1000,2000$; $r=0.05\times n,0.1\times n$;  $d = 0.25, 0.50, 1.00$.
\item Large: 12 instances. 3 with  each of the 4 combinations of the following parameters:  $m=5000,10000$; $n=1000$; $r=0.05\times n,0.1\times n$;  $d = 1.00$.
\end{itemize}

The numerical experiments with each category had different purposes. The tests with the `Small' instances have the main purpose of checking how tight are the bounds presented  in \cite{FampaLee2018ORL,XuFampaLee},  for the norms of the constructed matrices $H$. We note that this analysis requires the solution of the LP $P_1$. These are not easy LPs because they
are rather dense.
The tests with the `Medium' instances  have the main purpose of comparing the different local searches and initialization procedures that we have proposed. Finally,  the tests with the `Large' instances  have the main purpose of demonstrating the scalability  of our methodology.

\subsection{Selecting an initial block for the local search}
\label{secinit}


Our algorithm to construct the initial $r\times r$ non-singular submatrix of $A$ for our local searches is called `NSub', where NSub stands for non-singular submatrix. It comprises the Phase-One and Greedy algorithms described below.

In the Phase-One algorithm (see  Algorithm \ref{AlgPhaseOnerow}), we consider an $r\times n$  matrix $\tilde{I}_{\delta}$ with all elements equal to zero except the elements $[i,T(i)]$, for all $i=1,\ldots,r$, where $T$ is  a randomly selected set of $r$ indices from $\{1,\ldots,n\}$. The nonzero elements of $\tilde{I}_{\delta}$ are all equal to $\delta$, a parameter initially set to 1. Then, we define
$
\tilde{A}:=\left[
\begin{array}{c} 
\tilde{I}_{\delta}\\
A
\end{array}\right],
$
\begin{algorithm}[!ht]
	\footnotesize{
		\KwIn{$A\in \mathbb{R}^{m\times n}$, such that rank($A$)$=r$. }
		\KwOut{$S\subset M:=\{1,\ldots,m\}$,  such that   rank($A[S,:]$)$=|S|\leq r$, and $T\subset N:=\{1,\ldots,n\}$, such that $|T|=r$. }
$\delta:=1$; $\tilde{I}_{\delta}:=0_{r\times n}$\;
randomly select
$T\subset\{1,\ldots,n\}$, such that $|T|=r$\;
$S:=\{1,\ldots,r\}$\;
\While {$\delta>10^{-4} \;\; \& \;\; S\cap\{1,\ldots,r\} \neq \emptyset$} 
{
$\tilde{I}_{\delta}[i,T(i)]:=\delta$, for $i=1,\ldots,r$\;
$\tilde{A}:=\left[
\begin{array}{c} 
\tilde{I}_{\delta}\\
A
\end{array}\right]
$\;
$[S]:=$ FI(det)( $\tilde{A}[:,T]^\top$, $\tilde{A}[S,T]^\top$) (FI(det) is presented in Alg. \ref{AlgFI-p13})\;
$\delta:=\delta/10$\;
}
$S:=S\setminus\{1,\ldots,r\}$;  $S(i):=S(i)-r, \ \forall i$ \;

\caption{Algorithm Phase-One. \label{AlgPhaseOnerow} }
	}
\end{algorithm}
and iteratively apply the local search  `FI(det)' (presented in Algorithm \ref{AlgFI-p13}), to obtain a set $S$ of linearly-independent rows of $\tilde{A}[:,T]$, aiming at increasing the absolute value of the determinant of the submatrix $\tilde{A}[S,T]$. 
The local search is  initialized  at every iteration  with the transpose of an updated $r\times r$ non-singular submatrix $\tilde{A}[S,T]$. At the first iteration, we set $S=\{1,\ldots,r\}$, so  $\tilde{A}[S,T]$ 
 is the identity matrix.
 At each subsequent iteration, $S$ is updated with the solution of the local search, and the indices of $S$ still in $\{1,\ldots,r\}$  are  made less attractive to be in the next solution by decreasing   $\delta$ by a constant factor. The Phase-One algorithm stops when  $\delta$  becomes  $10^{-4}$ or when all the indices in $S$ are greater than $r$.

We execute the Phase-One algorithm up to a maximum number of times.  Each time, a set of $r$ columns $T$ of $\tilde{A}$ is randomly selected, and a set of $r$ rows $S$ is obtained. We then separate the indices from $S$ that correspond to rows of the matrix $A$, i.e., we set $S:=S\setminus\{1,\ldots,r\}$ and $S(i):=S(i)-r, \ \forall i$. We finally check if $|S|=r$; if so,  we stop  executing the algorithm  and output $A[S,T]$ as the $r\times r$ non-singular submatrix of $A$. 

We note that if all sets $T$ randomly selected in the executions of the Phase-One algorithm correspond 
to linearly-dependent columns of $A$, the final set $S$ will certainly contain less than $r$ indices.
In this case, we select from all the sets $S$ obtained in the executions of the Phase-One algorithm, the one with largest cardinality and starting from it,  we successively execute the Greedy algorithm  (see Algorithm \ref{AlgGreedyrow}). 
\begin{algorithm}
	\footnotesize{
		\KwIn{$A\in \mathbb{R}^{m\times n}$, such that rank($A$)$=r$, $\tau >0$, $S\subset M:=\{1,\ldots,m\}$,   such that  rank($A[S,:]$)$=|S|<r$. }
		\KwOut{$S\subset M$,  such that   rank($A[S,:]$)$=|S|\leq r$. }
\While{ $|S|<r$  }
{
Choose the least $i\in M\setminus S$  such that  $\sigma_{\min}(A[S\cup\{i\},:])>\tau$\;
$S:=S\cup\{i\}$\;
}

\caption{Algorithm Greedy. \label{AlgGreedyrow} }
	}
\end{algorithm}
At each execution, we iteratively add row indices to $S$. Each row selected is the first (with the least index), which keeps the minimum singular value of the partially constructed submatrix $A(S,:)$ greater than a given tolerance $\tau$. The iterations are repeated until no row is obtained or $|S|=r$. If $|S|<r$, $\tau$ is decreased by a constant factor and the algorithm is executed once more. We note that as $A$ has rank $r$, the convergence of this procedure is assured. We start the executions of the Greedy algorithm with $\tau$ slightly smaller than  the minimum singular value of $A(S,:)$ if the initial set $S$ is nonempty, or with $\tau=1$ otherwise.  

Finally, once we obtain the set $S$ with $r$ linearly-independent rows of $A$ with the Greedy algorithm, the last step of the NSub algorithm consists of obtaining $r$ linearly-independent columns of $A$. For that we rerun Algorithm \ref{AlgPhaseOnerow} (and also Algorithm \ref{AlgGreedyrow}, if necessary), but considering $A[S,:]^\top$ as the input matrix. In this case, as the input matrix has only $r$ columns,  the set $T$ in Algorithm \ref{AlgPhaseOnerow} is given by the indices of all columns, instead of being randomly selected. Therefore, the Phase-One algorithm is applied only once.

The NSub algorithm is depicted in Algorithm \ref{AlgNSub}.  Our Matlab implementation of  NSub is now available at \url{https://www.mathworks.com/matlabcentral/fileexchange/83638-linear-independent-rows-and-columns-generator} and can be used to either obtain $r$ linearly-independent rows of a given matrix $A$ with rank not smaller than $r$, or to compute an $r\times r$ non-singular submatrix of $A$. 

\begin{algorithm}
	\footnotesize{
		\KwIn{$A\in \mathbb{R}^{m\times n}$, such that rank($A$)$=r$, and $k_{\max}>0$}. 
		\KwOut{ $S\subset M:=\{1,\ldots,m\}$, $T\subset N:=\{1,\ldots,n\}$,  such that $|S|=|T|=r$, and $A[S,T]$ is non-singular. }
$S^1:=\emptyset$; $k:=1$\;		
\While{ $|S^k|<r \;\; \& \;\; k < k_{\max}$ }
{
$[S^{k+1},T]:=$ Algorithm Phase-One($A$)\;
$k:=k+1$\;
}
$S:=$argmax$_k\{|S^k|\}$\;
\If {$|S|<r$}
{
\If {$|S|>0$}{$\tau := \sigma_{\min}(A[S,:])/10$\;}
\Else{$\tau:=1$\;}

\While{$|S|<r$}
{
$[S]:=$Algorithm Greedy($A$, $S$)\;
$\tau:=\tau/10$\;
}
}
$[\bar{S}, \bar{T}]$:=Algorithm Phase-One($A[S,:]^\top$)\;
$T:=\bar{S}$\;
\If {$|T|>0$}{$\tau := \sigma_{\min}(A[S,T])/10$\;}
\Else{$\tau:=1$}
\While{$|T|<r$}
{
$[T]:=$Algorithm Greedy( $A[S,:]^\top$, $T$)\;
$\tau:=\tau/10$\;
}
\caption{Algorithm NSub. \label{AlgNSub} }
	}
\end{algorithm}

We note that we could  directly apply the Greedy algorithm to compute   the initial non-singular submatrix for the local searches, without calling the Phase-One algorithm. However, in our numerical experiments, we significantly improved the performance of NSub, when calling Phase-One as depicted in Algorithm \ref{AlgNSub}. We also observe  that our best numerical results were obtained by applying the NSub algorithm to $A$ when $m\geq n$, and to $A^\top$, otherwise. In other words, when $m<n$, we initially choose the linearly-independent columns of $A$.  

Given the submatrix computed by NSub, we perform a local search with the goal of reducing the 1-norm of the matrix $H$, by replacing rows and columns of the submatrix, as explained in the next subsection.

\subsection{The local-search procedures}
In Algorithm \ref{AlgLSFI} and \ref{AlgBI}, we present the local search
\begin{algorithm}[!ht]
	\footnotesize{
		\KwIn{ $A\in \mathbb{R}^{m\times n}$, with $r:=$rank($A$), \\$S\subset M:=\{1,\ldots,m\}$, $T\subset  N:=\{1,\ldots,n\}$,  such that $|S|=|T|=r$, and $A[S,T]$ is non-singular. }
		\KwOut{ possibly updated sets $S$, $T$. }		
$\tilde{A}:=A[S,T]$\;
$\bar{M}:=M\setminus S$,  $\bar{N}:=N\setminus T$,
$R:=A[\bar{M},T]$, $ C:=A[S,\bar{N}]$\;
$[L,U]:= LU(\tilde{A})$ (Compute the LU factorization of $\tilde{A}$)\;
$cont = true$\;
\While {($cont$)}
{
$cont=false$\;
\For {$\ell =1,\ldots, n-r$}
{
Solve $Ly=C[S,\ell]$, with solution $\hat{y}$\;
Solve $U\alpha=\hat{y}$, with solution $\hat{\alpha}$\;

\If {$|\hat{\alpha}|\nleq (1,\ldots,1)^\top$}
{
$\hat{\jmath}:=\min\{j :  |\hat{\alpha}_j|>1\}$ for `FI(det)', or $\hat{\jmath}:=\mbox{argmax}_j\{|\hat{\alpha}_j|\}$ for `FI$^+$(det)'\;
%
$aux:=\tilde{A}[S,\hat{\jmath}]$\;
$\tilde{A}[S,\hat{\jmath}] := C[S,\ell]$\;
$C[S,\ell]:=aux$\;
$T:=T\cup\{\bar{N}(\ell)\}\setminus \{T(\hat{\jmath})\}$\;
$\bar{N}:=\bar{N}\setminus\{\bar{N}(\ell)\}\cup \{T(\hat{\jmath})\}$\;
$[L,U]:= LU(\tilde{A})$ (update LU factorization of previous iteration)\;
$cont=true$\;
}
}
\For {$\ell =1,\ldots, m-r$}
{
Solve $U^\top y=R[\ell,T]^\top$, with solution $\hat{y}$\;
Solve $L^\top\alpha=\hat{y}$, with solution $\hat{\alpha}$\;
\If {$|\hat{\alpha}|\nleq (1,\ldots,1)^\top$}
{
$\hat{\jmath}:=\min\{j :  |\hat{\alpha}_j|>1\}$ for `FI(det)', or $\hat{\jmath}:=\mbox{argmax}_j\{|\hat{\alpha}_j|\}$ for `FI$^+$(det)' \;
$aux:=\tilde{A}[\hat{\jmath},T]$\;
$\tilde{A}[\hat{\jmath},T] := R[\ell,T]$\;
$R[\ell,T]:=aux$\;
$S:=S\cup\{\bar{M}(\ell)\}\setminus \{S(\hat{\jmath})\}$\;
$\bar{M}:=\bar{M}\setminus\{\bar{M}(\ell)\}\cup \{S(\hat{\jmath})\}$\;
$[L,U]:= LU(\tilde{A})$ (update LU factorization of previous iteration)\;
$cont=true$\;
}
}
}
\caption{`FI(det)' (`FI$^+$(det)')  for generalized inverses. \label{AlgLSFI} }
}
\end{algorithm}
procedures that consider as the criterion for improvement of the given solution, the increase in the absolute determinant of the  $r\times r$ non-singular submatrix of $A$.

Based on Theorem \ref{thm:approx}, for a given rank-$r$ matrix $A$,  the procedure starts from a set $S$ of $r$ rows and a set $T$ of $r$ columns of $A$, such that $A[S,T]$ is non-singular.

In the first loop of Algorithm \ref{AlgLSFI} (lines 7--18), a column of $A[S,N\setminus T]$ replaces a column of $A[S,T]$ if the absolute determinant increases with the replacement. To evaluate how much the determinant changes when each column of $A[S,T]$ is replaced by a given column $\gamma$  of $A[S,N\setminus T]$, we use the result in Remark \ref{remcol} (i.e.,
using Cramer's Rule).

\begin{jremark}\label{remcol}
Let $\gamma\in \mathbb{R}^n$ and $A\in \mathbb{R}^{n\times n}$ with $\det(A)\neq 0$. Let $A_{\gamma/j}$ be the matrix obtained by replacing the $j^{th}$ column of
$A$ by $\gamma$.  If $\hat{\alpha}\in \mathbb{R}^n$ solves the linear system of equations $A\alpha=\gamma$, then we have $\det(A_{\gamma/j})= \hat{\alpha}_j \times \det(A)$.\\
\end{jremark}

Similarly, in the second  loop of the algorithm (lines 19--30), a row of $A[M\setminus S, T]$ replaces a row of $A[S,T]$ if the absolute  determinant increases with the replacement. In this case, to evaluate how much the determinant changes when each row of $A[S,T]$ is replaced by a given row $\gamma$  of $A[M\setminus S, T]$, we use the equivalent result in Remark \ref{remrow}.

\begin{jremark}\label{remrow}
Let $\gamma\in \mathbb{R}^{1\times n}$ and $A\in \mathbb{R}^{n\times n}$ with $\det(A)\neq 0$. Let $A_{\gamma/i}$ be the matrix obtained by replacing the $i^{th}$ row of
$A$ by $\gamma$.  If $\hat{\alpha}\in \mathbb{R}^{1 \times n}$ solves the linear system of equations $\alpha A=\gamma$, then we have $\det(A_{\gamma/i})= \hat{\alpha}_i \times \det(A)$.\\
\end{jremark}

\noindent Use of Cramer's rule greatly improves the performance of these local searches.
\smallskip

Two algorithms,   `FI(det)' and `FI$^+$(det)', are presented in Algorithm  \ref{AlgLSFI}. The only differences between them are shown in lines 11 and 23.
For `FI$^+$(det)' (``first improvement plus''), we iteratively select a column (row) that is not in  $A[S,T]$ and exchange it with the column (row) of  $A[S,T]$ that leads to the greatest increase in the absolute determinant of the submatrix.  For `FI(det)' (``first improvement''),  the column (row) of $A[S,T]$ selected for the replacement is the one   of least index, that leads to an increase in the absolute  determinant.
\begin{algorithm}[!ht]
	\footnotesize{
		\KwIn{ $A\in \mathbb{R}^{m\times n}$, with $r:=$rank($A$), \\$S\subset M:=\{1,\ldots,m\}$, $T\subset  N:=\{1,\ldots,n\}$,  such that $|S|=|T|=r$, and $A[S,T]$ is non-singular. }
		\KwOut{ possibly updated sets $S$, $T$. }		
$\tilde{A}:=A[S,T]$\;
$\bar{M}:=M\setminus S$,  $\bar{N}:=N\setminus T$,
$R:=A[\bar{M},T]$, $ C:=A[S,\bar{N}]$\;
$[L,U]:= LU(\tilde{A})$ (Compute the LU factorization of $\tilde{A}$)\;
$biggest.\alpha_r=biggest.\alpha_c:=1$\;
$cont = true$\;
\While {($cont$)}
{
$cont=false$\;
\For {$\ell =1,\ldots, n-r$}
{
Solve $Ly=C[S,\ell]$, with solution $\hat{y}$\;
Solve $U\alpha=\hat{y}$, with solution $\hat{\alpha}$\;
$\hat{\alpha}_{\max} := \max_j\{|\hat{\alpha}_j|\}$\;
\If {$ \hat{\alpha}_{\max} >  biggest.\alpha_r$}
{
$biggest.\alpha_r:=\hat{\alpha}_{\max}$\;
$\hat{\jmath}_r:=\mbox{argmax}_j\{|\hat{\alpha}_j|\}$\;
$\hat{\ell}_r:=\ell$\;
}
}
\For {$\ell =1,\ldots, m-r$}
{
Solve $U^\top y=R[\ell,T]^\top$, with solution $\hat{y}$\;
Solve $L^\top\alpha=\hat{y}$, with solution $\hat{\alpha}$\;
$\hat{\alpha}_{\max} := \max_j\{|\hat{\alpha}_j|\}$\;
\If {$ \hat{\alpha}_{\max} >  biggest.\alpha_c$}
{
$biggest.\alpha_c:=\hat{\alpha}_{\max}$\;
$\hat{\jmath}_c:=\mbox{argmax}_j\{|\hat{\alpha}_j|\}$\;
$\hat{\ell}_c:=\ell$\;
}
}
\If {$\max\{biggest.\alpha_r,biggest.\alpha_c\} > 1$}
{
$cont=true$\;
\If {$biggest.\alpha_r > biggest.\alpha_c$}
{
$aux:=\tilde{A}[S,\hat{\jmath}_r]$\;
$\tilde{A}[S,\hat{\jmath}_r] := C[S,\hat{\ell}_r]$\;
$C[S,\hat{\ell}_r]:=aux$\;
$T:=T\cup\{\bar{N}(\hat{\ell}_r)\}\setminus \{T(\hat{\jmath}_r)\}$\;
$\bar{N}:=\bar{N}\setminus\{\bar{N}(\hat{\ell}_c)\}\cup \{T(\hat{\jmath}_r)\}$\;
$[L,U]:= LU(\tilde{A})$ (update LU factorization of previous iteration)\;
}
\Else
 {
$aux:=\tilde{A}[\hat{\jmath}_c,T]$\;
$\tilde{A}[\hat{\jmath}_c,T] := R[\hat{\ell}_c,T]$\;
$R[\hat{\ell}_c,T]:=aux$\;
$S:=S\cup\{\bar{M}(\ell)\}\setminus \{S(\hat{\jmath})\}$\;
$\bar{M}:=\bar{M}\setminus\{\bar{M}(\hat{\ell}_c)\}\cup \{S(\hat{\jmath}_c)\}$\;
$[L,U]:= LU(\tilde{A})$ (update LU factorization of previous iteration)\;
 }
 }
}
\caption{`BI(det) for generalized inverses'. \label{AlgBI} }
}
\end{algorithm}

We also present in Algorithm \ref{AlgBI}, the algorithm `BI(det)' (``best improvement''), where the pair of rows or columns exchanged at each iteration is selected as the pair that leads to the greatest increase in the absolute determinant, among all possibilities.

Algorithms `FI(det)', `FI$^+$(det)', and `BI(det)' stop when no replacement of a row or column of $A[S,T]$ would lead to an increase in the absolute determinant, i.e., when we reach a  local   maximizer for the absolute determinant, according to Definition \ref{def2018}.
\begin{algorithm}[!ht]
	\footnotesize{
		\KwIn{ $A\in \mathbb{R}^{m\times n}$, with $r:=$rank($A$), \\$S\subset M:=\{1,\ldots,m\}$, $T\subset  N:=\{1,\ldots,n\}$,  such that $|S|=|T|=r$, and $A[S,T]$ is non-singular. }
		\KwOut{ possibly updated sets $S$, $T$. }		
$\tilde{A}:=A[S,T]$\;
$\bar{M}:=M\setminus S$,  $\bar{N}:=N\setminus T$,
$R:=A[\bar{M},T]$, $ C:=A[S,\bar{N}]$\;
$cont = true$\;
\While {($cont$)}
{
$cont=false$\;
$B=A[S,T]$\;
$Binv:=B^{-1}$\;
$nBinv=\|Binv\|_1$\;
\For {$\ell =1,\ldots, n-r$}
{
$\gamma:=A[S,\bar{N}(\ell)]$\;
\For {$j =1,\ldots, r$}
{
Let $B_{\gamma/j}$ be the matrix obtained by replacing the $j^{th}$ column of
$B$ by $\gamma$\;
$nBinv^+:=\|B_{\gamma/j}^{-1}\|_1$ (Computed with the result in Remark \ref{remnormcol})\;
\If {$nBinv^+ < nBinv$}
{
$B := B_{\gamma/j}$\;
$Binv := B_{\gamma/j}^{-1}$ (Computed with the result in Remark \ref{remnormcol})\;
$nBinv=nBinv^+$\;
$T:=T\cup\{\bar{N}(\ell)\}\setminus \{T(j)\}$\;
$\bar{N}:=N\setminus T$\;
$cont=true$\;
break \;
}
}
}
$B=A[S,T]^\top$\;
$Binv:=B^{-1}$\;
$nBinv=\|Binv\|_1$\;
\For {$\ell =1,\ldots, m-r$}
{
$\gamma:=A[\bar{M}(\ell),T]^\top$\;
\For {$j =1,\ldots, r$}
{
Let $B_{\gamma/j}$ be the matrix obtained by replacing the $j^{th}$ column of
$B$ by $\gamma$\;
$nBinv^+:=\|B_{\gamma/j}^{-1}\|_1$ (Computed with the result in Remark \ref{remnormcol})\;
\If {$nBinv^+ < nBinv$}
{
$B := B_{\gamma/j}$\;
$Binv := B_{\gamma/j}^{-1}$ (Computed with the result in Remark \ref{remnormcol})\;
$nBinv=nBinv^+$\;
$S:=S\cup\{\bar{M}(\ell)\}\setminus \{S(j)\}$\;
$\bar{M}:=M\setminus S$\;
$cont=true$\;
break\;
}
}
}
}
\caption{`FI(norm)' for generalized inverses. \label{AlgnormLSFI} }
}
\end{algorithm}

Algorithm \ref{AlgnormLSFI} represents the local search  `FI(norm)'.  In this case, we  consider as the criterion for improvement of the given solution, the decrease in the 1-norm of $H$, or equivalently, the decrease in the 1-norm of the inverse of the  $r\times r$ non-singular submatrix of $A$ being considered.
To evaluate how much the 1-norm of the inverse of the submatrix  changes when each column (row) of $A[S,T]$ is replaced by a given column (row) $\gamma$  of $A[S,N\setminus T]$ ($A[M\setminus S, T]$), we use the result in Remark \ref{remnormcol}.

\vbox{
\begin{jremark}\label{remnormcol}
Let $\gamma\in \mathbb{R}^n$ and $A:=(a_1,  \ldots, a_j\mathbin{,}\ldots, a_n) \in \mathbb{R}^{n\times n}$ with $\det(A)\neq 0$. Let $A_{\gamma/j}$ be the matrix obtained by replacing the $j^{th}$ column of
$A$ by $\gamma$, and
 $v=(v_1, \ldots, v_j, \ldots, v_n)^\top := A^{-1}\gamma$. If $v_j\neq 0$, define
 \[
\bar{v}:=\left(-\frac{v_1}{v_j}, \ldots, -\frac{v_{j-1}}{v_j},\  \frac{1}{v_j},\  -\frac{v_{j+1}}{v_j}, \ldots, -\frac{v_n}{v_j}\right)^\top.
\]
Then
\[
A_{\gamma/j}^{-1} = \Theta \ A^{-1},
\]
where
\[
\Theta = (e_1, \ldots, e_{j-1},\  \bar{v},\ e_{j+1},\ldots, e_n),
\]
and $e_i$ are the standard unit vectors.
\end{jremark}
}

\noindent Use of Remark \ref{remnormcol} greatly improves the performance of these local searches.
\smallskip


\begin{jremark}
In Algorithms \ref{AlgLSFI} and \ref{AlgBI} (and later in Algorithms \ref{AlgFI-p13} and \ref{AlgBI-p13}), we need to
update LU factorizations of an $r\times r$ matrix $B$ under low-rank changes. Practical and numerically-stable algorithms for 
LU factorizations employ ``partial or complete pivoting'', and Matlab provides this functionality, 
calculating such factorizations in $\mathcal{O}(r^3)$ floating-point operations (in the dense case).
But Matlab does not have functionality for efficiently updating these
factorizations, while in theory they can be updated in  $\mathcal{O}(r^2)$ floating-point operations (in the dense case); see 
for example, \cite{Sahinidis} or \cite{Gondzio}). In principle, we do advocate a proper updating approach, but we computed
our new LU factorizations (with partial pivoting) from scratch each time for two reasons: (i) the updating procedures are
not available in Matlab, and (ii) our algorithms turn out to be very fast even without performing fast LU updates.
\end{jremark}


\subsection{Numerical results} We initially consider the experiments done with the 90 instances in the `Small' category, which had the main purpose of analyzing the ratios between 
the 1-norm of the matrices $H$ computed by the three local searches based on the determinant, with the minimum 1-norm of a generalized inverse given by the solution of the LP problem $P_1$  $(\|H\|_1/z_{P_1})$. We aim at checking how close these ratios are from the upper bound given by Theorem \ref{thm:approx}.

In   Figure \ref{barratplp1dens},  we present the average ratios   for the matrices with the same dimension, rank, and density.  From Theorem \ref{thm:approx}, we know that these ratios cannot be greater than $r^2$, and we see from the results, that for the matrices considered in our tests, we stay quite far from this upper bound (even though the upper bound is the best possible). In general, the ratios increase with the rank $r$, the dimension $m=n$, and the density $d$  of the matrices,  but even for $r=50$, we  obtain ratios less than 2. So, our conclusion is that the
worst-case bound, while best possible, is extremely pessimistic.

\begin{figure}[!ht]
\centering
\includegraphics[scale=0.5]{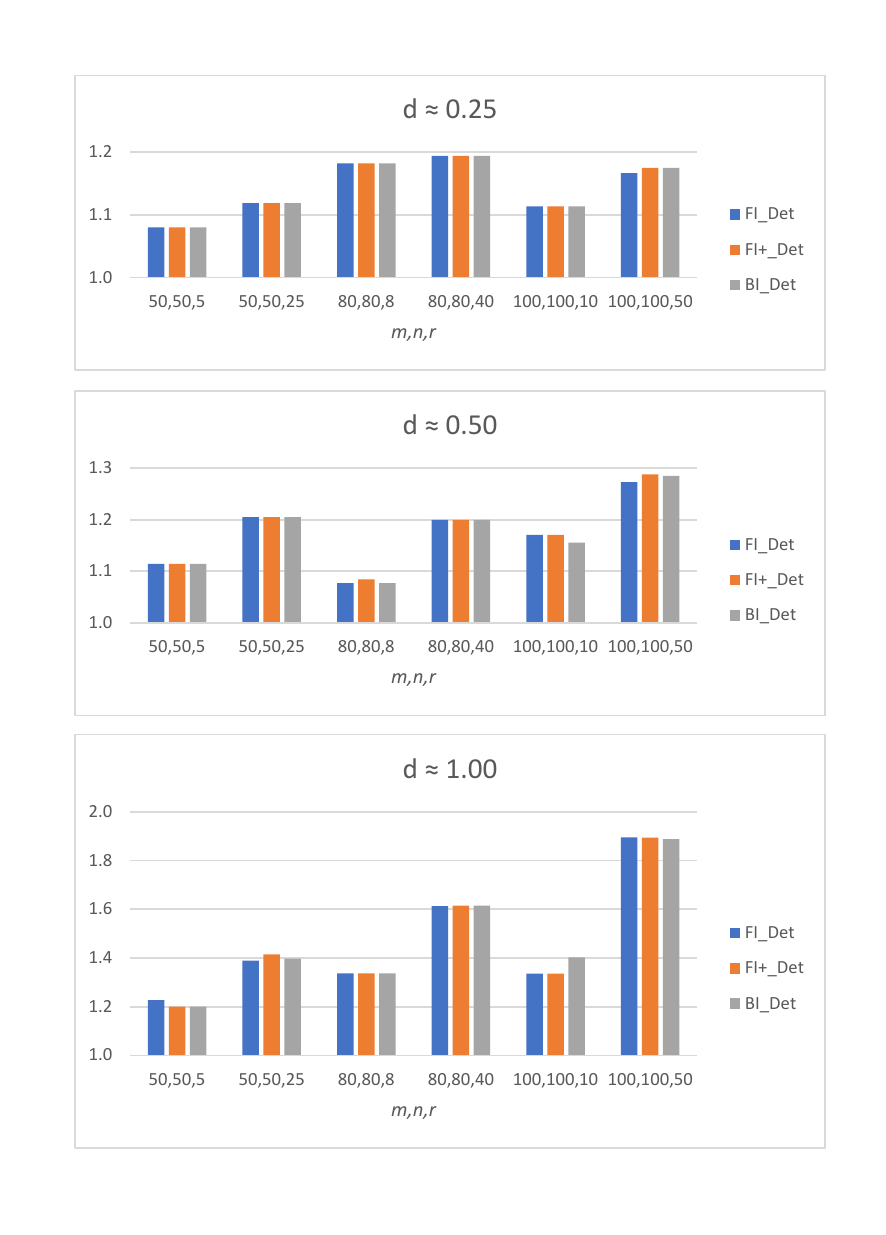}
\caption{ $\|H\|_1/z_{P_1}$ (Small) (generalized inverse)\label{barratplp1dens}}
\end{figure}

In Table \ref{tab:1}, besides presenting the average ratios depicted in Figure \ref{barratplp1dens}, we also  present the average running time to compute the generalized inverses. In case of the local searches, the total time to compute the generalized inverse is given by the sum of the time to generate the initial matrix $H$ by the 
NSub algorithm (Algorithm \ref{AlgNSub}), and the time of the local search (FI(det), FI$^+$(det), or BI(det)).

\begin{table}[!ht]
	\centering
	\tiny
\begin{tabular}{r|lll|rrlll}
		\hline
		&\multicolumn{3}{|c|}{ $\|H\|_1/z_{P_1}$}&\multicolumn{5}{|c}{Time(sec)}\\
		\multicolumn{1}{c|}{$m,r,d$}& \multicolumn{1}{|c}{FI(det)}   & \multicolumn{1}{c}{FI$^+$(det)} & \multicolumn{1}{c|}{BI(det)} & \multicolumn{1}{|c}{$P_1$} &   \multicolumn{1}{c}{NSub} &\multicolumn{1}{c}{FI(det)}   & \multicolumn{1}{c}{FI$^+$(det)} & \multicolumn{1}{c}{BI(det)} \\
		\hline
		50,05,0.25  & 1.080 & 1.080 & 1.080 & 0.495   & 0.091 & 0.004 & 0.003 & 0.005 \\
		50,25,0.25  & 1.119 & 1.119 & 1.119 & 4.711   & 0.094 & 0.003 & 0.002 & 0.004 \\
		80,08,0.25  & 1.182 & 1.182 & 1.182 & 1.734   & 0.077 & 0.003 & 0.003 & 0.006 \\
		80,40,0.25  & 1.195 & 1.195 & 1.195 & 27.284  & 0.149 & 0.009 & 0.006 & 0.020 \\
		100,10,0.25 & 1.114 & 1.114 & 1.114 & 3.297   & 0.090 & 0.004 & 0.003 & 0.011 \\
		100,50,0.25 & 1.167 & 1.175 & 1.175 & 65.156  & 0.266 & 0.005 & 0.004 & 0.016 \\
		50,05,0.50  & 1.115 & 1.115 & 1.115 & 0.489   & 0.065 & 0.004 & 0.004 & 0.005 \\
		50,25,0.50  & 1.205 & 1.205 & 1.205 & 4.995   & 0.084 & 0.004 & 0.003 & 0.007 \\
		80,08,0.50  & 1.077 & 1.085 & 1.077 & 1.705   & 0.068 & 0.005 & 0.004 & 0.008 \\
		80,40,0.50  & 1.200 & 1.200 & 1.200 & 29.604  & 0.131 & 0.006 & 0.004 & 0.013 \\
		100,10,0.50 & 1.171 & 1.171 & 1.155 & 3.806   & 0.094 & 0.005 & 0.004 & 0.010 \\
		100,50,0.50 & 1.273 & 1.288 & 1.285 & 69.810  & 0.288 & 0.007 & 0.004 & 0.021 \\
		50,05,1.00  & 1.228 & 1.201 & 1.201 & 0.687   & 0.003 & 0.007 & 0.006 & 0.010 \\
		50,25,1.00  & 1.390 & 1.416 & 1.399 & 6.691   & 0.036 & 0.008 & 0.005 & 0.014 \\
		80,08,1.00  & 1.337 & 1.337 & 1.337 & 2.391   & 0.018 & 0.005 & 0.004 & 0.011 \\
		80,40,1.00  & 1.614 & 1.615 & 1.615 & 43.598  & 0.088 & 0.010 & 0.006 & 0.029 \\
		100,10,1.00 & 1.337 & 1.337 & 1.403 & 4.792   & 0.020 & 0.006 & 0.006 & 0.018 \\
		100,50,1.00 & 1.896 & 1.895 & 1.889 & 124.856 & 0.139 & 0.030 & 0.017 & 0.089 \\
		\hline
	\end{tabular}
	\caption{Local Searches for generalized inverse vs. $P_1$}\label{tab:1}
\end{table}

We see from Figure \ref{barratplp1dens} and  Table \ref{tab:1}, that  the three local searches converge to solutions of similar quality on most of the experiments. 

We observe in Table \ref{tab:1} that the running times to solve the LP $P_1$ increase quickly with the dimension of the matrix, and are much higher than the times for the local searches. Therefore, we can already see that the LP $P_1$ is not useful as a computational alternative to
our local searches when we consider larger instances (and additionally, as we have mentioned,
the solutions produced by the LP do not have the reflexive property, nor are they
nicely block structured).

Next, we consider the experiments done with the 360 instances in the `Medium' category, which had the main purpose of  comparing the different local searches  proposed.  
 We  present in Table \ref{tab:2} average results for each group of 30 instances with the same configuration, described in the first column. In the next three columns we present statistics for the local searches based on the determinant, which are initialized  with the solutions given by the NSub  algorithm, and in the last three columns we consider the application of the local searches based on the 1-norm of  $H$, which are initialized with  the solutions given by the three first local searches. In the first half of the table,  we show the mean and standard deviation of the 
 relative difference between  the 1-norm of the matrix $H$ obtained by each local search and the minimum value  among all of them, denoted by $||H_{best}||_1$. $H_{best}$ is naturally obtained with the application of one of the local searches based on the 1-norm. 

	\begin{table}[!ht]
		\centering
		\tiny
	\begin{tabular}{l|c|c|c|c|c|c}
			\hline
			\multicolumn{1}{c|}{$m,r,d$}& \multicolumn{1}{|c|}{FI(det)}   & \multicolumn{1}{|c|}{FI$^+$(det)} & \multicolumn{1}{|c|}{BI(det)} & \multicolumn{1}{|c|}{FI(det)}   & \multicolumn{1}{c}{FI$^+$(det)} & \multicolumn{1}{|c}{BI(det)} \\
			&   & & & \multicolumn{1}{|c|}{FI(norm)}   & \multicolumn{1}{c}{FI(norm)} & \multicolumn{1}{|c}{FI(norm)} \\
			\hline
			&\multicolumn{6}{|c}{ $(\|H\|_1 - \|H_{best}\|_1)/\|H_{best}\|_1$ }\vphantom{$\Sigma^{R}$} \\[3pt]
			\hline
			1000,050,0.25 & 0.073/0.034 & 0.072/0.034 & 0.071/0.036 & 0.000/0.000 & 0.000/0.000 & 0.000/0.000 \\
			1000,050,0.50 & 0.098/0.031 & 0.100/0.030 & 0.097/0.030 & 0.000/0.001 & 0.000/0.001 & 0.001/0.002 \\
			1000,050,1.00 & 0.086/0.026 & 0.086/0.026 & 0.086/0.026 & 0.001/0.004 & 0.001/0.004 & 0.000/0.002 \\
			1000,100,0.25 & 0.134/0.024 & 0.132/0.025 & 0.133/0.026 & 0.000/0.001 & 0.000/0.001 & 0.001/0.002 \\
			1000,100,0.50 & 0.089/0.047 & 0.088/0.047 & 0.088/0.048 & 0.000/0.003 & 0.000/0.000 & 0.000/0.000 \\
			1000,100,1.00 & 0.132/0.042 & 0.131/0.046 & 0.131/0.047 & 0.002/0.004 & 0.002/0.004 & 0.001/0.003 \\
			2000,100,0.25 & 0.116/0.035 & 0.115/0.035 & 0.114/0.036 & 0.000/0.002 & 0.000/0.001 & 0.000/0.000 \\
			2000,100,0.50 & 0.171/0.036 & 0.170/0.035 & 0.171/0.034 & 0.001/0.002 & 0.001/0.003 & 0.001/0.003 \\
			2000,100,1.00 & 0.128/0.054 & 0.130/0.057 & 0.129/0.054 & 0.002/0.006 & 0.002/0.006 & 0.002/0.005 \\
			2000,200,0.25 & 0.202/0.050 & 0.210/0.065 & 0.213/0.060 & 0.005/0.010 & 0.004/0.005 & 0.007/0.011 \\
			2000,200,0.50 & 0.160/0.046 & 0.161/0.045 & 0.161/0.044 & 0.001/0.003 & 0.001/0.004 & 0.001/0.003 \\
			2000,200,1.00 & 0.217/0.046 & 0.219/0.047 & 0.220/0.048 & 0.003/0.007 & 0.002/0.005 & 0.002/0.005\\
			\hline
			&\multicolumn{6}{|c}{ Time(sec) }\\
			\hline
			1000,050,0.25 & 0.116/0.035 & 0.109/0.032 & 0.469/0.116   & 1.888/0.426    & 1.899/0.439    & 1.843/0.381    \\
			1000,050,0.50 & 0.267/0.041 & 0.241/0.050 & 1.111/0.337   & 13.806/2.242   & 14.163/2.440   & 14.097/2.442   \\
			1000,050,1.00 & 0.372/0.061 & 0.348/0.091 & 1.755/0.268   & 29.225/5.751   & 28.888/5.567   & 28.990/5.817   \\
			1000,100,0.25 & 1.320/0.293 & 0.975/0.273 & 8.871/3.565   & 317.976/59.536 & 314.073/60.679 & 319.427/52.886 \\
			1000,100,0.50 & 0.104/0.032 & 0.110/0.030 & 0.369/0.133   & 1.995/0.322    & 1.983/0.336    & 1.994/0.330    \\
			1000,100,1.00 & 0.246/0.041 & 0.236/0.064 & 1.112/0.529   & 14.737/2.585   & 14.865/2.729   & 14.697/2.646   \\
			2000,100,0.25 & 0.337/0.054 & 0.318/0.093 & 1.459/0.390   & 30.938/5.014   & 30.503/5.153   & 30.220/5.091   \\
			2000,100,0.50 & 1.301/0.258 & 1.070/0.206 & 8.537/5.039   & 312.795/60.716 & 323.046/67.272 & 320.428/67.722 \\
			2000,100,1.00 & 0.113/0.043 & 0.132/0.042 & 0.771/0.536   & 2.303/0.562    & 2.340/0.620    & 2.281/0.608    \\
			2000,200,0.25 & 0.279/0.083 & 0.258/0.063 & 2.215/1.311   & 20.197/4.851   & 21.013/4.596   & 20.791/4.553   \\
			2000,200,0.50 & 0.290/0.105 & 0.319/0.088 & 1.945/2.843   & 38.609/7.940   & 38.806/8.670   & 38.700/8.634   \\
			2000,200,1.00 & 1.153/0.454 & 0.912/0.320 & 10.513/13.156 & 395.193/62.373 & 416.49/88.601  & 398.360/67.740 \\
			\hline
		\end{tabular}
		\caption{Local Searches for generalized inverse (Mean/Std Dev) (Medium) }\label{tab:2}
	\end{table}	

Comparing to the tests with the `Small' instances,  we see that on this larger group of instances of higher dimension,  
the solutions obtained by the three local searches based on the determinant are still of similar quality. Consequently, the solutions obtained by the  searches based on the 1-norm are also of similar quality. By applying these 1-norm searches, we are able to improve the solutions from determinant searches by approximately 7 to 22\%.
Furthermore, we see that this  improvement  comes with  a  high computational cost.  The necessity of computing the inverse of the $r\times r$ submatrix of $A$ at each iteration, significantly increases the time of these searches. Even though we use the result in Remark \ref{remnormcol} to accelerate this computation, it  still makes the norm searches slower than the determinant  searches. We finally note that the 1-norm searches are able to improve more the solutions from the determinant searches as the rank and the dimension of the matrix increase. 

In Table \ref{tab:3}, we present the number of swaps for each local search.  Combining these results with the running time of the procedures, 
we conclude that, despite the fact that the best improvement is commonly pointed as a good criterion for local searches in the literature, in our case,  BI(det)  could be discarded. Comparing it to the other determinant searches, we see that, although it converges to solutions of similar quality  performing fewer swaps, it is much more time consuming. Comparing the two other searches based on the determinant, FI$^+$(det) preforms slightly better, with a smaller number of swaps and the average computational time a bit smaller. We can also observe the high cost of the swaps performed by the searches based on the 1-norm. 
 These observations are pointed out in Figure  \ref{lscomp},  where we show the relation between the average relative 1-norm difference  to the minimum norm obtained by all searches  and the average running time of the local searches
 for the larger instances in the `Medium'  category. The hollow circle indicates that the local search had worse average result and longer average running time than another procedure, and therefore, should not be adopted. We note that we use a logarithmic scale for the running time. Once again we see some  improvement given by the norm searches, but with a very high computational cost.

\begin{table}[!ht]
	\centering
	\tiny
\begin{tabular}{r|rrrrrr}
		\hline
		\multicolumn{1}{c|}{$m,r,d$}& \multicolumn{1}{|c}{FI(det)}   & \multicolumn{1}{c}{FI$^+$(det)} & \multicolumn{1}{c}{BI(det)} & \multicolumn{1}{c}{FI(det)}   & \multicolumn{1}{c}{FI$^+$(det)} & \multicolumn{1}{c}{BI(det)} \\
		&   & &  & \multicolumn{1}{c}{FI(norm)}   & \multicolumn{1}{c}{FI(norm)} & \multicolumn{1}{c}{FI(norm)} \\
		\hline
		1000,050,0.25 & 27.833  & 23.367  & 12.800 & 14.300  & 14.000  & 13.900  \\
		1000,050,0.50 & 46.300  & 36.467  & 20.433 & 31.600  & 32.033  & 31.400  \\
		1000,050,1.00 & 42.333  & 35.367  & 18.333 & 31.867  & 31.900  & 32.000  \\
		1000,100,0.25 & 79.767  & 60.233  & 33.467 & 75.433  & 74.600  & 74.833  \\
		1000,100,0.50 & 20.133  & 16.933  & 8.867  & 17.033  & 16.933  & 16.933  \\
		1000,100,1.00 & 41.900  & 32.367  & 20.633 & 37.600  & 37.767  & 37.600  \\
		2000,100,0.25 & 32.200  & 26.933  & 13.600 & 40.300  & 40.267  & 39.967  \\
		2000,100,0.50 & 67.600  & 54.300  & 34.867 & 87.933  & 88.033  & 87.967  \\
		2000,100,1.00 & 116.667 & 73.633  & 26.467 & 23.767  & 23.967  & 24.267  \\
		2000,200,0.25 & 180.833 & 104.067 & 50.933 & 60.800  & 62.033  & 61.233  \\
		2000,200,0.50 & 103.700 & 59.700  & 21.200 & 57.067  & 57.167  & 56.600  \\
		2000,200,1.00 & 139.700 & 81.300  & 43.400 & 116.267 & 116.767 & 116.833 \\
		\hline
	\end{tabular}
	\caption{Number of swaps (Medium) (generalized inverse) }\label{tab:3}
\end{table}

\begin{figure}[!ht]
\centering
\includegraphics[width=\textwidth]{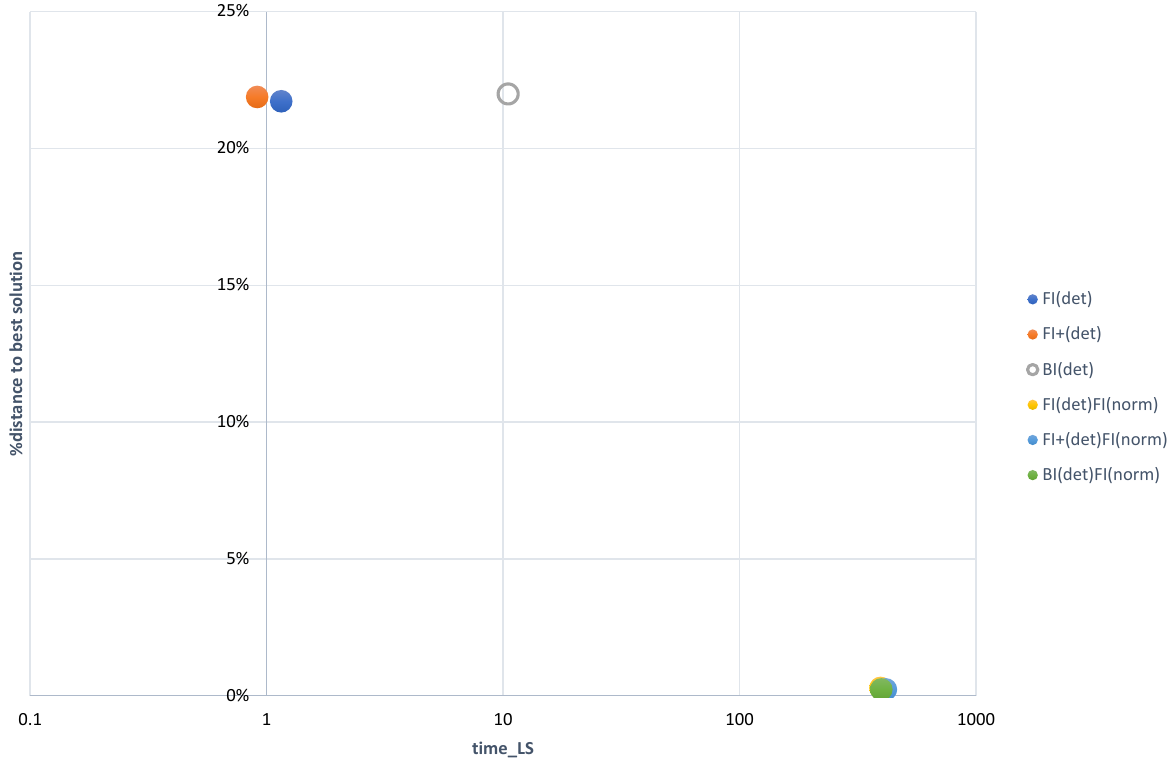}
\caption{ Local searches ($m=2000$, $r=200$, $d=1$) (generalized inverse)}\label{lscomp}
\end{figure}

Finally, our last experiments intend to show the scalability of the procedures, considering matrices with up to 10000 rows, 1000 columns, rank up to 100, and density equal to one. As the BI(det) procedure was not successful in the previous test, we did not run it on the `Large'  category. In Table \ref{tab:5} we see that the average
relative differences between  the 1-norm of the matrix $H$ obtained by each local search based on the determinant and the minimum 1-norm  among the four local searches
are approximately between 9 and 15\%. So, comparing to best solutions found, the searches based on the determinant keep the same quality observed on the smaller instances. Furthermore, these larger matrices are obtained in less than 5 seconds on average. The NSub algorithm had a very good performance on these large instances, being able to compute the initial non-singular submatrices for the local searches in less than 2 seconds on average.  

\begin{table}[!ht]
	\centering
	\tiny
\begin{tabular}{r|rrrrr}
		\hline
		& \multicolumn{5}{|c}{$(\|H\|_1 - \|H_{best}\|_1)/\|H_{best}\|_1$ } \\
		\multicolumn{1}{c|}{$m,n,r$}   &\multicolumn{1}{|c}{FI(det)}   & \multicolumn{1}{c}{FI$^+$(det)} &\multicolumn{1}{c}{FI(det)}   & \multicolumn{1}{c}{FI$^+$(det)} &\\
		&   & &\multicolumn{1}{c}{FI(norm)}   & \multicolumn{1}{c}{FI(norm)} &\\
		\hline
		5000,1000,050  & 0.129 & 0.152 & 0.034 & 0.000 \\
		5000,1000,100  & 0.098 & 0.143 & 0.006 & 0.038 \\
		10000,1000,050 & 0.088 & 0.108 & 0.001 & 0.040 \\
		10000,1000,100 & 0.104 & 0.090 & 0.000 & 0.015 \\
		\hline
		&\multicolumn{5}{|c}{ Time(sec) }\\
		\multicolumn{1}{c|}{$m,n,r$}   &\multicolumn{1}{|c}{FI(det)}   & \multicolumn{1}{c}{FI$^+$(det)} &\multicolumn{1}{c}{FI(det)}   & \multicolumn{1}{c}{FI$^+$(det)} & \multicolumn{1}{c}{NSub}\\
		&   & &\multicolumn{1}{c}{FI(norm)}   & \multicolumn{1}{c}{FI(norm)} &\\
		\hline
		5000,1000,050  & 0.541 & 0.578 & 9.614   & 14.551  & 0.641 \\
		5000,1000,100  & 1.288 & 1.461 & 94.748  & 68.547  & 1.139 \\
		10000,1000,050 & 1.088 & 1.351 & 10.882  & 10.800  & 1.897 \\
		10000,1000,100 & 2.623 & 2.512 & 166.830 & 131.578 & 1.942 \\
		\hline
	\end{tabular}
	\caption{Local searches (Large) (generalized inverse)}\label{tab:5}
\end{table}



\section{ah-symmetric generalized inverse}\label{sec:ahsymginv}

Recall the key use of an ah-symmetric generalized inverse:
 if $H$ is an ah-symmetric generalized inverse of $A$, then $\hat{x}:=Hb$ solves $\min\{\|Ax-b\|_2:~x\in\mathbb{R}^n\}$.
The local-search procedures for the ah-symmetric reflexive generalized inverse  are based on the block construction procedure presented in \cite{XuFampaLee}. More specifically, on   Theorem \ref{thm:ahconstruction}, Definition \ref{defahconstruction}, and Theorem \ref{thmrwithP3}, presented next.

\begin{theorem}[\cite{XuFampaLee}]\label{thm:ahconstruction}
For $A\in\mathbb{R}^{m\times n}$, let $r := \rank(A)$. For any $T$, an ordered subset of $r$ elements from $\{1,\dots,n\}$, let $\hat{A}:=A[:,T]$ be the $m \times r$ submatrix of $A$ formed by columns $T$. If $\rank(\hat{A})=r$, let
$
\hat{H}
\vrule height 10pt depth 0pt width 0pt
 := \hat{A}^{\dagger} =(\hat{A}^\top\hat{A})^{-1}\hat{A}^\top.
$
The $n \times m$ matrix $H$ with all rows equal to zero, except rows $T$, which are given by $\hat{H}$, is an ah-symmetric reflexive generalized inverse of $A$.
\end{theorem}

In the context of the least-square problem, such a ``column block solution''
amounts to choosing a set of $r$ ``explanatory variables'' in the context of multicolinearity
(i.e., dependent columns of $A$), which is highly desirable in terms of explainability.
It remains to choose a good column block solution, by which we mean having entries under control
(via approximate 1-norm minimization).

\begin{definition}[\cite{XuFampaLee}]\label{defahconstruction}
Let $A$ be an arbitrary $m\times n$, rank-$r$ matrix, and let $S$ be an ordered subset of $r$ elements from $\{1,\dots,m\}$ such that these $r$ rows of $A$ are linearly independent. For $T$ an ordered subset of $r$ elements from $\{1,\dots,n\}$, and fixed $\epsilon\ge0$, if $|\det(A[S,T])|$ cannot be increased by a factor of more than $1+\epsilon$ by swapping an element of $T$ with one from its complement, then we say that $A[S,T]$ is a $(1+\epsilon)$-local maximizer for the absolute determinant on the set of $r\times r$ non-singular submatrices of $A[S,:]$.
\end{definition}

\begin{theorem}[\cite{XuFampaLee}]\label{thmrwithP3}
Let $A$ be an arbitrary $m\times n$, rank-$r$ matrix, and let $S$ be an ordered subset of $r$ elements from $\{1,\dots,m\}$ such that these $r$ rows of $A$ are linearly independent. Choose $\epsilon\ge0$, and let $\tilde{A}:=A[S,T]$ be a $(1+\epsilon)$-local maximizer for the absolute determinant on the set of $r\times r$ non-singular submatrices of $A[S,:]$. Then the $n \times m$ matrix $H$ constructed by Theorem \ref{thm:ahconstruction} over $\hat{A}:=A[:,T]$, is an ah-symmetric reflexive generalized inverse of $A$ satisfying $\|H\|_1\le r(1+\epsilon)\|H_{opt}^{ah,r}\|_1$, where $H_{opt}^{ah,r}$ is a $1$-norm minimizing ah-symmetric reflexive generalized inverse of $A$.
\end{theorem}

As before, the $\epsilon$ of Definition \ref{defahconstruction} and Theorem \ref{thmrwithP3} is used
in \cite{FampaLee2018ORL} to gain polynomial running time in $1/\epsilon$. For the purpose of actual computations,
our observation has been that $\epsilon$ can be chosen to be zero.
We further note that in \cite{XuFampaLee}, we demonstrated that the
bound in Theorem \ref{thm:approx} is the best possible. However, we
will see in our experiments that the bound is overly pessimistic
by a wide margin.

The idea of the algorithms considered in this section is to select an $m\times r$ rank-$r$ submatrix of $A$, and construct an ah-symmetric reflexive generalized inverse of $A$ with the M-P pseudoinverse of this submatrix, as described in Theorem \ref{thm:ahconstruction}.

Next, we give details of the algorithms and present  numerical results.
The  test matrices used in the computational experiments are the same 462 matrices considered in the previous section, and the same $r\times r$ non-singular submatrices of $A$ constructed by  the NSub algorithm,
were used to initialize the local-search procedures discussed in this section.

To analyze the local-search procedures proposed, we  compare their solutions to the solutions of LP problem identified below as $P_{123}$. Its solutions corresponds to $\|H_{opt}^{ah,r}\|_1$. As defined in Theorem \ref{thmrwithP3},   $H_{opt}^{ah,r}$ is a $1$-norm minimizing ah-symmetric \emph{reflexive} generalized inverse of $A$.
In order to formulate $P_{123}$ as an LP problem, we linearize \ref{property2}, using the following result.

\begin{proposition}(see, for example, \cite[Proposition 4.3]{FFL2016})
\label{prop1}
If $H$ satisfies \ref{property1} and \ref{property3}, then $AH=AA^{\dagger}$.
\end{proposition}
Therefore, if  $H$ satisfies \ref{property1} and \ref{property3},
then \ref{property2} can be reformulated as the linear equation
\begin{equation}\label{p2lin}
HAA^{\dagger}=H.
\end{equation}
Considering \eqref{p2lin}, we then have

\[
\begin{array}{lll}
(P_{123}) \ z_{P_{123}}:=&\min  & \left\langle J , T \right\rangle~,  \\
& \mbox{s.t.:}
             &T - H \geq 0~,\\
		     &&T + H \geq 0~,\\
		     &&AHA=A~,\\
		     &&(AH)^{\top}=(AH)~,\\
                  &&HAA^{\dagger}=H~.
\end{array}
\]


It is important to note that in the case of a generalized inverse, we
could only compare the 1-norm quality of our solution to the
optimal value of the  LP $P_1$, ignoring
the reflexivity condition \ref{property2}. Here, we can compare
to the optimal value of th LP $P_{123}$ because \ref{property2} can be linearized
when \ref{property3} is imposed.

\subsection{The local-search procedures}

In Algorithms \ref{AlgFI-p13} and  \ref{AlgBI-p13}, we present the local-search procedures that consider as the criterion for improvement of the given solution, the increase in the absolute determinant of the current  $r\times r$ non-singular submatrix of $A$. Based on Theorem \ref{thmrwithP3},  the procedures start from a set $S$ of $r$ rows and a set $T$ of $r$ columns of $A$, such that $A[S,T]$ is non-singular. We note that unlike  what is done in Algorithms \ref{AlgLSFI} and  \ref{AlgBI}, in Algorithms \ref{AlgFI-p13} and \ref{AlgBI-p13} only columns of $A[S,T]$ are considered to be exchanged in order to increase the determinant. From the result in Theorem \ref{thmrwithP3}, we see that \emph{any} set $S$ of $r$ linearly-independent rows of $A$ could be used in the search for a local maximizer for the absolute determinant on the set of $r\times r$ non-singular submatrices of $A[S,:]$. The initial submatrix $A[S,T]$ is obtained by the NSub algorithm (Algorithm \ref{AlgNSub}).  

\begin{algorithm}[!ht]
	\footnotesize{
		\KwIn{ $A\in \mathbb{R}^{m\times n}$, with $r:=$rank($A$), \\$S\subset M:=\{1,\ldots,m\}$, $T\subset  N:=\{1,\ldots,n\}$,  such that $|S|=|T|=r$, and $A[S,T]$ is non-singular. }
		\KwOut{ possibly updated set $T$. }		
$\tilde{A}:=A[S,T]$\;
$\bar{N}:=N\setminus T$\;
$ C:=A[S,\bar{N}]$\;
$[L,U]:= LU(\tilde{A})$ (Compute the LU factorization of $\tilde{A}$)\;
$cont= true$\;
\While {($cont$)}
{
$cont=false$\;
\For {$\ell =1,\ldots, n-r$}
{
Solve $Ly=C[S,\ell]$, with solution $\hat{y}$\;
Solve $U\alpha=\hat{y}$, with solution $\hat{\alpha}$\;
$[\hat{\alpha}^{\ell}_{\max},\hat{\jmath}]:=\max_{i}\{|\hat{\alpha}_i|\}$\;
\If {$|\hat{\alpha}|\nleq (1,\ldots,1)^\top$}
{
$\hat{\jmath}:=\min\{j :  |\hat{\alpha}_j|>1\}$ for `FI(det)', or $\hat{\jmath}:=\mbox{argmax}_j\{|\hat{\alpha}_j|\}$ for `FI$^+$(det)'\;
$aux:=\tilde{A}[S,\hat{\jmath}]$\;
$\tilde{A}[S,\hat{\jmath}] := C[S,\ell]$\;
$C[S,\ell]:=aux$\;
$T:=T\cup\{\bar{N}(\ell)\}\setminus \{T(\hat{\jmath})\}$\;
$\bar{N}:=\bar{N}\setminus\{\bar{N}(\ell)\}\cup \{T(\hat{\jmath})\}$\;
$[L,U]:= LU(\tilde{A})$ (update LU factorization of previous iteration)\;
$cont=true$\;
}
}
}
\caption{`FI(det)' (`FI$^+$(det)')  for ah-symmetric generalized inverses  \label{AlgFI-p13} }
}
\end{algorithm}
\begin{algorithm}
	\footnotesize{
\KwIn{ $A\in \mathbb{R}^{m\times n}$, with $r:=$rank($A$), \\$S\subset M:=\{1,\ldots,m\}$, $T\subset  N:=\{1,\ldots,n\}$,  such that $|S|=|T|=r$, and $A[S,T]$ is non-singular. }
\KwOut{ possibly updated set $T$. }		
$\tilde{A}:=A[S,T]$\;
$\bar{N}:=N\setminus T$\;
$ C:=A[S,\bar{N}]$\;
$[L,U]:= LU(\tilde{A})$ (Compute the LU factorization of $\tilde{A}$)\;
$biggest.\alpha_r:=1$\;
$cont:= true$\;
\While {($cont$)}
{
$cont:=false$\;
\For {$\ell =1,\ldots, n-r$}
{
Solve $Ly=C[S,\ell]$, with solution $\hat{y}$\;
Solve $U\alpha=\hat{y}$, with solution $\hat{\alpha}$\;
$\hat{\alpha}_{\max} := \max_j\{|\hat{\alpha}_j|\}$\;
\If {$ \hat{\alpha}_{\max} >  biggest.\alpha_r$}
{
$biggest.\alpha_r:=\hat{\alpha}_{\max}$\;
$\hat{\jmath}_r:=\mbox{argmax}_j\{|\hat{\alpha}_j|\}$\;
$\hat{\ell}_r:=\ell$\;
}
}
\If {$biggest.\alpha_r > 1$}
{
$cont:=true$\;
$aux:=\tilde{A}[S,\hat{\jmath}_r]$\;
$\tilde{A}[S,\hat{\jmath}_r] := C[S,\hat{\ell}_r]$\;
$C[S,\hat{\ell}_r]:=aux$\;
$T:=T\cup\{\bar{N}(\hat{\ell}_r)\}\setminus \{T(\hat{\jmath}_r)\}$\;
$\bar{N}:=\bar{N}\setminus\{\bar{N}(\hat{\ell}_c)\}\cup \{T(\hat{\jmath}_r)\}$\;
$[L,U]:= LU(\tilde{A})$ (update LU factorization of previous iteration)\;
}

}
\caption{`BI(det)'   for ah-symmetric generalized inverses  \label{AlgBI-p13} }
}
\end{algorithm}

We present in Algorithm  \ref{AlgnormLSFIp3} the `FI(norm)' for ah-symmetric  generalized inverses. The algorithm is obtained by  excluding from Algorithm \ref{AlgnormLSFI}, the loop where row exchanges are performed, and also by replacing the inverses of the matrices $B$ and $B_{\gamma/j}$, with their pseudoinverses.

To evaluate how much the 1-norm of the M-P pseudoinverse of the submatrix changes when each column of $A[:,T]$ is replaced by a given column $\gamma$  of $A[:,N\setminus T]$, we use the result in Remark \ref{remnormcolpinv}.

\begin{jremark}\label{remnormcolpinv}(see, for example, \cite{meyer1973generalized})
Let $\gamma = Av \in \mathbb{R}^m$ and $A:=(a_1,  \ldots, a_j\mathbin{,}\ldots\mathbin{,} a_r) \mathbin{\in} \mathbb{R}^{m\times r}$ with $\rank(A)=r$. Let $A_{\gamma/j}$ be the matrix obtained by replacing the $j^{th}$ column of
$A$ by $\gamma$. If $v_j\neq 0$, define \[
\bar{v}:=\left(-\frac{v_1}{v_j}, \ldots, -\frac{v_{j-1}}{v_j},\  \frac{1}{v_j},\  -\frac{v_{j+1}}{v_j}, \ldots, -\frac{v_r}{v_j}\right)^\top.
\]
Then
\[
A_{\gamma/j}^{\dagger} = \Theta \ A^{\dagger},
\]
where
\[
\Theta = (e_1, \ldots, e_{j-1},\  \bar{v}, \ e_{j+1},\ldots, e_r),
\]
and $e_i$ are the standard unit vectors.
\end{jremark}
\begin{proof}
Notice that $A_{\gamma/j}=A\Theta^{-1}$. We could verify that
$$
A_{\gamma/j}^\dagger A_{\gamma/j} =  I_r, ~A_{\gamma/j}A_{\gamma/j}^\dagger  = AA^\dagger ,
$$
which implies $A_{\gamma/j}^\dagger $ is the M-P pseudoinverse of $A_{\gamma/j}$.
\qed
\end{proof}

\begin{algorithm}[!ht]
	\footnotesize{
		\KwIn{ $A\in \mathbb{R}^{m\times n}$, with $r:=$rank($A$), \\ $T\subset  N:=\{1,\ldots,n\}$,  such that $|T|=r$, and $A[:,T]$ has rank $r$. }
		\KwOut{ possibly updated set $T$. }		
$\bar{N}:=N\setminus T$\;
$B=A[:,T]$\;
$Bpinv:=(B^\top B)^{-1}B^\top$\;
$nBpinv=\|Bpinv\|_1$\;
$cont = true$\;
\While {($cont$)}
{
$cont=false$\;
\For {$\ell =1,\ldots, n-r$}
{
$\gamma:=A[:,\bar{N}(\ell)]$\;
\For {$j =1,\ldots, r$}
{
Let $B_{\gamma/j}$ be the matrix obtained by replacing the $j^{th}$ column of
$B$ by $\gamma$\;
$Bpinv^+:= (B_{\gamma/j}^\top B_{\gamma/j})^{-1}B_{\gamma/j}^\top$\;
$nBpinv^+:=\|Bpinv^+\|_1$\;
\If {$nBpinv^+ < nBpinv$}
{
$B := B_{\gamma/j}$\;
$Bpinv := Bpinv^+$\;
$nBpinv=nBpinv^+$\;
$T:=T\cup\{\bar{N}(\ell)\}\setminus \{T(j)\}$\;
$\bar{N}:=N\setminus T$\;
$cont=true$\;
break \;
}
}
}
}
\caption{`FI(norm)' for ah-symmetric generalized inverses. \label{AlgnormLSFIp3} }
}
\end{algorithm}


\subsection{Numerical results} \label{results_ahrg}
Similarly to the previous section, we initially consider the experiments done with the 90 instances in the `Small' category, with the purpose of analyzing the ratios between the
the 1-norm of the matrices $H$ computed by the three local searches based on the determinant, with the minimum 1-norm of a ah-symmetric generalized inverse given by the solution of the LP problem $P_{123}$  $(\|H\|_1/z_{P_{123}})$. We now aim at checking how close these ratios are from the upper bound given by Theorem \ref{thmrwithP3}.

In   Figure \ref{barratplp1densP3},  we present the average ratios   for the matrices with the same dimension, rank, and density.  From Theorem \ref{thmrwithP3}, we know that these ratios cannot be greater than $r$, and we also see from the results, that for the matrices considered in our tests, we stay quite far from this upper bound. In general, the ratios increase with the rank $r$, the dimension $m=n$, and the density $d$  of the matrices,  but even for $r=50$, we  obtain ratios smaller than 1.5.

\begin{figure}[!ht]
\centering
\includegraphics[scale=0.5]{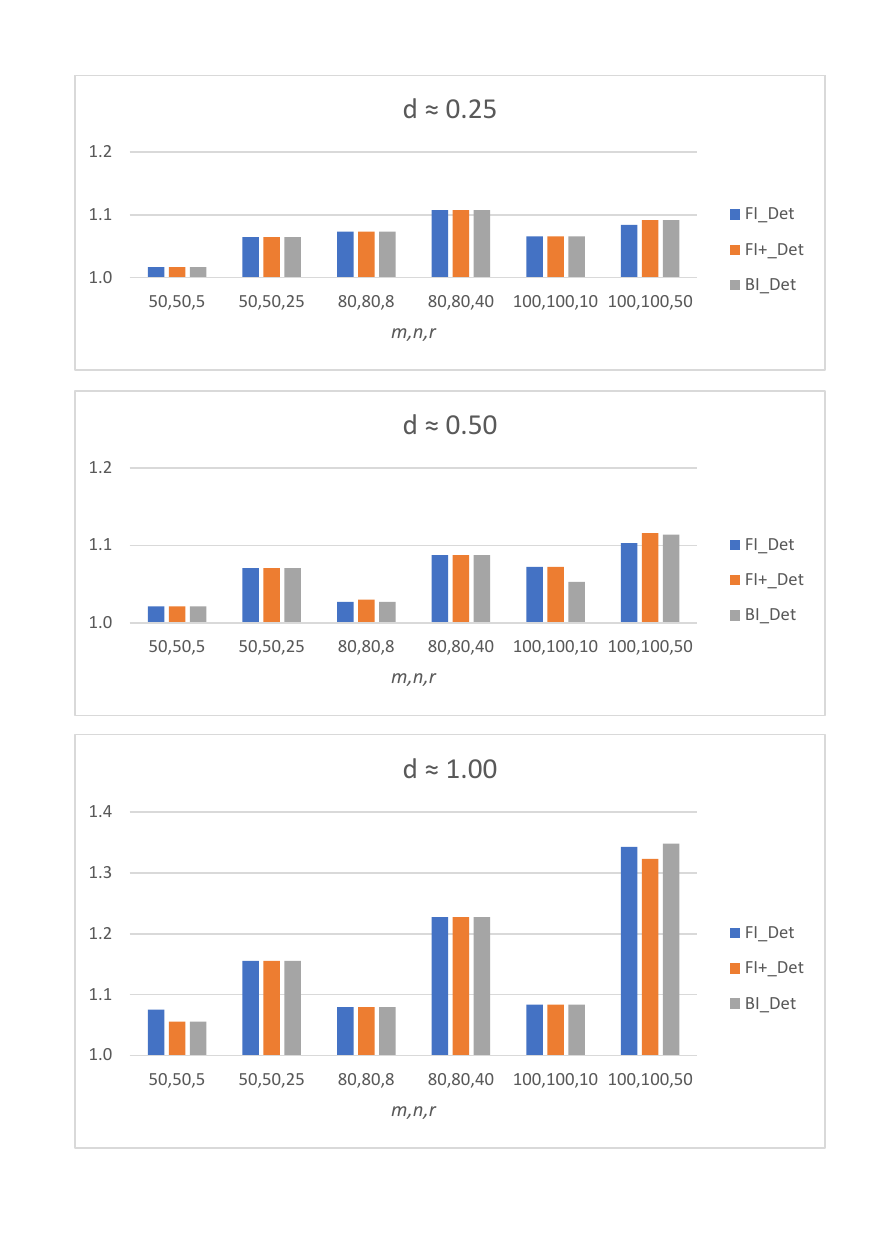}
\caption{ $\|H\|_1/z_{P_{123}}$ (Small) (ah-symmetric generalized inverse)}\label{barratplp1densP3}
\end{figure}

In Table \ref{tab:1P3}, besides presenting the average ratios depicted in Figure \ref{barratplp1densP3}, we also  present the average running time to compute the generalized inverses. In case of the local searches, the total time to compute the generalized inverse is given by the sum of the time to generate the initial matrix $H$ by the NSub algorithm (Algorithm \ref{AlgNSub}), and the time of the local search (FI(det), FI$^+$(det), or BI(det)).

\begin{table}[!ht]
	\centering
	\tiny
\begin{tabular}{r|lll|rrlll}
		\hline
		&\multicolumn{3}{|c|}{ $\|H\|_1/z_{P_{123}}$}&\multicolumn{5}{|c}{Time(sec)}\\
		\multicolumn{1}{c|}{$m,r,d$}& \multicolumn{1}{|c}{FI(det)}   & \multicolumn{1}{c}{FI$^+$(det)} & \multicolumn{1}{c|}{BI(det)} & \multicolumn{1}{|c}{$P_{123}$} &   \multicolumn{1}{c}{NSub} &\multicolumn{1}{c}{FI(det)}   & \multicolumn{1}{c}{FI$^+$(det)} & \multicolumn{1}{c}{BI(det)} \\
		\hline
		50,05,0.25  & 1.017 & 1.017 & 1.017 & 1.526    & 0.091 & 0.005 & 0.004 & 0.005 \\
		50,25,0.25  & 1.065 & 1.065 & 1.065 & 5.506    & 0.094 & 0.002 & 0.001 & 0.001 \\
		80,08,0.25  & 1.073 & 1.073 & 1.073 & 5.325    & 0.077 & 0.002 & 0.002 & 0.003 \\
		80,40,0.25  & 1.108 & 1.108 & 1.108 & 34.729   & 0.149 & 0.003 & 0.003 & 0.006 \\
		100,10,0.25 & 1.066 & 1.066 & 1.066 & 11.695   & 0.090 & 0.003 & 0.003 & 0.006 \\
		100,50,0.25 & 1.084 & 1.092 & 1.092 & 93.004   & 0.266 & 0.007 & 0.005 & 0.011 \\
		50,05,0.50  & 1.022 & 1.022 & 1.022 & 1.807    & 0.065 & 0.002 & 0.001 & 0.002 \\
		50,25,0.50  & 1.071 & 1.071 & 1.071 & 7.873    & 0.084 & 0.002 & 0.002 & 0.003 \\
		80,08,0.50  & 1.027 & 1.030 & 1.027 & 8.427    & 0.068 & 0.002 & 0.002 & 0.004 \\
		80,40,0.50  & 1.088 & 1.088 & 1.088 & 102.428  & 0.131 & 0.005 & 0.003 & 0.008 \\
		100,10,0.50 & 1.073 & 1.073 & 1.053 & 24.495   & 0.094 & 0.003 & 0.003 & 0.005 \\
		100,50,0.50 & 1.103 & 1.116 & 1.114 & 481.614  & 0.288 & 0.011 & 0.006 & 0.023 \\
		50,05,1.00  & 1.075 & 1.056 & 1.056 & 2.699    & 0.003 & 0.004 & 0.003 & 0.005 \\
		50,25,1.00  & 1.155 & 1.155 & 1.155 & 41.476   & 0.036 & 0.004 & 0.003 & 0.007 \\
		80,08,1.00  & 1.079 & 1.079 & 1.079 & 32.213   & 0.018 & 0.003 & 0.002 & 0.007 \\
		80,40,1.00  & 1.227 & 1.227 & 1.227 & 692.151  & 0.088 & 0.009 & 0.005 & 0.019 \\
		100,10,1.00 & 1.084 & 1.084 & 1.084 & 169.194  & 0.020 & 0.004 & 0.004 & 0.012 \\
		100,50,1.00 & 1.343 & 1.323 & 1.348 & 3672.983 & 0.139 & 0.017 & 0.010 & 0.040 \\
		\hline
	\end{tabular}
	\caption{Local Searches for ah-symmetric generalized inverse vs. $P_{123}$}\label{tab:1P3}
\end{table}

We see from Figure \ref{barratplp1densP3} and  Table \ref{tab:1P3}, that  the three local searches converge to solutions of similar quality on most of the experiments. 
We also see in Table \ref{tab:1P3} that the running times to solve $P_{123}$ increase quickly  with the dimension of the matrix, and are much higher than the times for the local searches.

Next, we consider the experiments done with the 360 instances in the `Medium' category, which had the main purpose of comparing the different local searches  proposed.  
 We  present in Table \ref{tab:2P3} average results for each group of 30 instances with the same configuration, described in the first column. In the next three columns, we present statistics for the local searches based on the determinant, which are initialized  with the solutions given by the NSub algorithm, and in the last three columns we consider the application of the local searches based on the 1-norm of  $H$, which are initialized with  the solutions given by the three first local searches. In the first half of the table,  we show the mean and standard deviation of the 
 relative difference between  the 1-norm of the matrix $H$ obtained by each local search and the minimum value  among all of them, denoted by $||H_{best}||_1$.

\begin{table}[!ht]
	\centering
	\tiny
\begin{tabular}{l|c|c|c|c|c|c}
		\hline
		\multicolumn{1}{c|}{$m,r,d$}& \multicolumn{1}{|c|}{FI(det)}   & \multicolumn{1}{|c|}{FI$^+$(det)} & \multicolumn{1}{|c|}{BI(det)} & \multicolumn{1}{|c|}{FI(det)}   & \multicolumn{1}{c}{FI$^+$(det)} & \multicolumn{1}{|c}{BI(det)} \\
		&   & & & \multicolumn{1}{|c|}{FI(norm)}   & \multicolumn{1}{c}{FI(norm)} & \multicolumn{1}{|c}{FI(norm)} \\
		\hline
		&\multicolumn{6}{|c}{ $(\|H\|_1 - \|H_{best}\|_1)/\|H_{best}\|_1$ }\vphantom{$\Sigma^{R}$} \\[3pt]
		\hline
		1000,050,0.25 & 0.029/0.026 & 0.029/0.026 & 0.029/0.026 & 0.000/0.000 & 0.000/0.000 & 0.000/0.000 \\
		1000,050,0.50 & 0.045/0.022 & 0.045/0.022 & 0.044/0.022 & 0.000/0.002 & 0.000/0.002 & 0.000/0.000 \\
		1000,050,1.00 & 0.042/0.022 & 0.041/0.020 & 0.041/0.020 & 0.000/0.000 & 0.000/0.000 & 0.000/0.000 \\
		1000,100,0.25 & 0.064/0.017 & 0.064/0.017 & 0.064/0.018 & 0.000/0.000 & 0.000/0.001 & 0.001/0.004 \\
		1000,100,0.50 & 0.042/0.028 & 0.042/0.028 & 0.042/0.028 & 0.000/0.000 & 0.000/0.000 & 0.000/0.000 \\
		1000,100,1.00 & 0.064/0.028 & 0.066/0.030 & 0.066/0.030 & 0.001/0.003 & 0.001/0.003 & 0.001/0.003 \\
		2000,100,0.25 & 0.056/0.020 & 0.056/0.020 & 0.056/0.020 & 0.000/0.000 & 0.000/0.000 & 0.000/0.000 \\
		2000,100,0.50 & 0.082/0.026 & 0.083/0.027 & 0.083/0.027 & 0.000/0.001 & 0.001/0.001 & 0.000/0.001 \\
		2000,100,1.00 & 0.052/0.030 & 0.054/0.032 & 0.053/0.029 & 0.001/0.003 & 0.002/0.007 & 0.003/0.006 \\
		2000,200,0.25 & 0.089/0.037 & 0.095/0.041 & 0.095/0.035 & 0.002/0.005 & 0.004/0.007 & 0.006/0.010 \\
		2000,200,0.50 & 0.062/0.017 & 0.062/0.016 & 0.061/0.017 & 0.001/0.003 & 0.000/0.001 & 0.000/0.001 \\
		2000,200,1.00 & 0.098/0.024 & 0.100/0.025 & 0.100/0.029 & 0.002/0.005 & 0.001/0.004 & 0.002/0.004 \\
		\hline
		&\multicolumn{6}{|c}{ Time(sec) }\\
		\hline
		1000,050,0.25 & 0.051/0.010 & 0.048/0.005 & 0.079/0.059 & 8.505/2.592      & 8.338/2.798      & 8.220/2.626      \\
		1000,050,0.50 & 0.119/0.040 & 0.098/0.023 & 0.215/0.229 & 32.498/10.200    & 32.487/10.583    & 32.339/10.694    \\
		1000,050,1.00 & 0.205/0.047 & 0.199/0.026 & 0.325/0.115 & 148.958/37.366   & 147.906/38.710   & 145.106/36.993   \\
		1000,100,0.25 & 0.643/0.205 & 0.490/0.107 & 1.559/1.769 
	& 995.64/257.79  & 978.29/227.30  & 991.33/232.10  \\
		1000,100,0.50 & 0.049/0.012 & 0.048/0.009 & 0.076/0.071 & 14.223/3.432     & 14.289/3.758     & 14.151/3.722     \\
		1000,100,1.00 & 0.129/0.043 & 0.106/0.024 & 0.397/0.310 & 55.756/12.981    & 56.009/13.701    & 56.973/12.637    \\
		2000,100,0.25 & 0.176/0.048 & 0.201/0.032 & 0.350/0.174 & 233.947/35.658   & 234.675/36.255   & 236.703/39.453   \\
		2000,100,0.50 & 0.594/0.184 & 0.441/0.155 & 2.975/3.004 & 
		1467.39/275.77 & 1445.03/311.96 & 1454.70/312.04 \\
		2000,100,1.00 & 0.069/0.030 & 0.058/0.019 & 0.332/0.239 & 23.153/6.454     & 22.469/6.044     & 22.610/6.376     \\
		2000,200,0.25 & 0.216/0.077 & 0.155/0.045 & 1.219/0.706 & 85.532/13.034    & 89.602/18.831    & 87.646/20.069    \\
		2000,200,0.50 & 0.208/0.092 & 0.238/0.060 & 1.050/1.472 & 349.917/75.253   & 348.767/75.847   & 350.218/76.554   \\
		2000,200,1.00 & 0.651/0.348 & 0.447/0.175 & 4.683/6.211 
& 2085.29/345.92 & 1974.42/361.11 & 1990.26/350.17 \\
		 		\hline
	\end{tabular}
	\caption{Local Searches for ah-symmetric generalized inverse (Mean/Std Dev) (Medium) }\label{tab:2P3}
\end{table}	

Comparing to the tests with the `Small' instances,  we note that we  still have  solutions of similar quality obtained by the three local searches based on the determinant, for this larger group of instances of higher dimension.
 The average relative difference between the norms of the solutions obtained by these  searches and  the minimum norms    approximately goes from 3 to 10\%,
increasing with  the rank and the dimension. 
We also see that the  improvement on the solutions found by the local searches based on the 1-norm of $H$, when compared to the determinant searches comes once more with  a  high computational cost.


In Table \ref{tab:3P3}, we present the number of swaps for each local search. Combining these results with the running time of the procedures, 
we conclude that  the BI(det) procedure could be discarded. Comparing it to the other determinant searches, we see again that, although it converges to solutions of similar quality  performing fewer swaps, it is much more time consuming. We can also observe the high cost of the swaps performed by the searches based on the 1-norm.
 This observation is illustrated in Figure  \ref{lscompP3}, where we show the relation between the average relative 1-norm difference  to the minimum norm obtained by all searches  and the average running time of the local searches for the larger instances in the `Medium'  category. The hollow circle indicates  that the local search had worse average result and longer average running time than another procedure, and therefore, should not be adopted. We note that we use a logarithmic scale for the running time. Once more we see some improvement given by the norm searches, but with a very high computational cost.

\begin{table}[!ht]
	\centering
	\tiny
\begin{tabular}{r|rrrrrr}
		\hline
		\multicolumn{1}{c|}{$m,r,d$}& \multicolumn{1}{|c}{FI(det)}   & \multicolumn{1}{c}{FI$^+$(det)} & \multicolumn{1}{c}{BI(det)} & \multicolumn{1}{c}{FI(det)}   & \multicolumn{1}{c}{FI$^+$(det)} & \multicolumn{1}{c}{BI(det)} \\
		&   & &  & \multicolumn{1}{c}{FI(norm)}   & \multicolumn{1}{c}{FI(norm)} & \multicolumn{1}{c}{FI(norm)} \\
		\hline
		1000,050,0.25 & 4.133   & 3.667   & 2.833  & 6.267  & 6.267  & 6.267  \\
		1000,050,0.50 & 8.700   & 6.800   & 5.067  & 14.733 & 14.733 & 14.700 \\
		1000,050,1.00 & 5.633   & 4.967   & 2.833  & 16.200 & 16.100 & 16.100 \\
		1000,100,0.25 & 15.033  & 12.033  & 9.133  & 35.367 & 35.233 & 35.167 \\
		1000,100,0.50 & 4.867   & 4.233   & 2.833  & 7.500  & 7.500  & 7.500  \\
		1000,100,1.00 & 18.833  & 15.533  & 11.967 & 17.633 & 17.900 & 17.867 \\
		2000,100,0.25 & 5.433   & 4.900   & 3.100  & 19.767 & 19.767 & 19.667 \\
		2000,100,0.50 & 31.500  & 27.333  & 21.700 & 44.067 & 43.800 & 43.833 \\
		2000,100,1.00 & 113.267 & 70.700  & 24.867 & 10.067 & 10.033 & 9.600  \\
		2000,200,0.25 & 177.067 & 100.800 & 48.867 & 23.767 & 24.533 & 24.467 \\
		2000,200,0.50 & 99.867  & 56.567  & 19.833 & 23.600 & 23.400 & 22.933 \\
		2000,200,1.00 & 134.500 & 76.933  & 41.267 & 52.333 & 52.400 & 51.900 \\
		\hline
	\end{tabular}
	\caption{Number of swaps (Medium) (ah-symmetric generalized inverse) }\label{tab:3P3}
\end{table}

\begin{figure}[!ht]
\centering
\includegraphics[width=\textwidth]{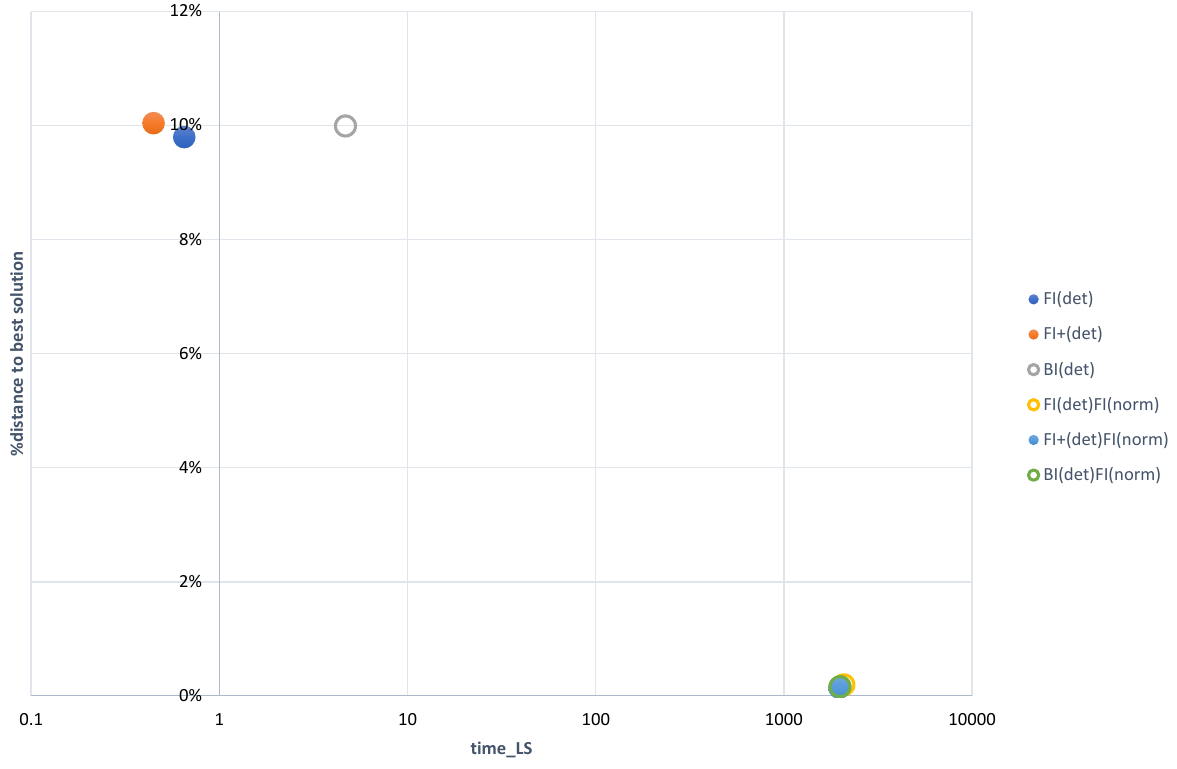}
\caption{ Local searches ($m=2000$,$r=200$,$d=1$) (ah-symmetric generalized inverse)}\label{lscompP3}
\end{figure}

Finally, our last experiments intend to show the scalability of the procedures, considering matrices with up to 10000 rows, 1000 columns, rank up to 100, and density equal to one. As the BI(det) procedure, was not successful in the previous test, we did not run it on the `Large'  category. In Table \ref{tab:5P3}, we see that the average
 relative differences between  the 1-norm of the matrix $H$ obtained by each local search based on the determinant and the minimum value  among all the four local searches 
are approximately between 3 and 6\%. Comparing to best solutions found, the searches based on the determinant keep the quality observed on the smaller instances. Furthermore, theses matrices are obtained in less than 2.5 seconds on average, even for the largest instances.

\begin{table}[!ht]
	\centering
	\tiny
\begin{tabular}{r|rrrrr}
		\hline
		& \multicolumn{5}{|c}{$(\|H\|_1 - \|H_{best}\|_1)/\|H_{best}\|_1$ } \\
		\multicolumn{1}{c|}{$m,n,r$}   &\multicolumn{1}{|c}{FI(det)}   & \multicolumn{1}{c}{FI$^+$(det)} &\multicolumn{1}{c}{FI(det)}   & \multicolumn{1}{c}{FI$^+$(det)} &\\
		&   & &\multicolumn{1}{c}{FI(norm)}   & \multicolumn{1}{c}{FI(norm)} &\\
		\hline
		5000,1000,050  & 0.053 & 0.045 & 0.020 & 0.021 \\
		5000,1000,100  & 0.034 & 0.040 & 0.011 & 0.024 \\
		10000,1000,050 & 0.045 & 0.059 & 0.000 & 0.025 \\
		10000,1000,100 & 0.030 & 0.028 & 0.012 & 0.005 \\
		\hline
		&\multicolumn{5}{|c}{ Time(sec) }\\
		\multicolumn{1}{c|}{$m,n,r$}   &\multicolumn{1}{|c}{FI(det)}   & \multicolumn{1}{c}{FI$^+$(det)} &\multicolumn{1}{c}{FI(det)}   & \multicolumn{1}{c}{FI$^+$(det)} & \multicolumn{1}{c}{NSub}\\
		&   & &\multicolumn{1}{c}{FI(norm)}   & \multicolumn{1}{c}{FI(norm)} &\\
		\hline
		5000,1000,050  & 0.100 & 0.109 & 221.678 & 157.442  & 0.641 \\
		5000,1000,100  & 0.252 & 0.291 & 534.947 & 395.332  & 1.139 \\
		10000,1000,050 & 0.084 & 0.090 & 424.338 & 307.216  & 1.897 \\
		10000,1000,100 & 0.317 & 0.343 & 987.229 & 1145.040 & 1.942 \\
		\hline
	\end{tabular}
	\caption{Local searches (Large) (ah-symmetric generalized inverse)}\label{tab:5P3}
\end{table}

\section{Symmetric generalized inverse}\label{sec:symginv}

Now we assume that $A$ is symmetric, and we are interested in finding
a good symmetric reflexive generalized inverse.
The local-search procedures for the symmetric reflexive  generalized inverse  are based on the block construction procedure presented in \cite{XuFampaLee}. More specifically, on   Theorem \ref{thm:symconstruction}, Definition \ref{defsymconstruction}, and Theorem \ref{thm:symapprox}, presented next.

\begin{theorem}[\cite{XuFampaLee}]\label{thm:symconstruction}
For a symmetric matrix $A\in\mathbb{R}^{n\times n}$, let $r := \rank(A)$. Let $\tilde{A}:=A[S]$ be any $r \times r$
non-singular principal submatrix of $A$. Let $H\in\mathbb{R}^{n\times n}$ be equal to zero, except its submatrix with row/column indices $S$ is equal to $\tilde{A}^{-1}$. Then $H$ is a symmetric reflexive generalized inverse of $A$.
\end{theorem}

\begin{definition}[\cite{XuFampaLee}]\label{defsymconstruction}
Let $A$ be an arbitrary $n\times n$, rank-$r$ matrix. For $S$ an ordered subset of $r$ elements from $\{1,\dots,n\}$ and fixed $\epsilon\ge0$, if $|\det(A[S])|>0$ cannot be increased by a factor of more than $1+\epsilon$ by swapping an element of $S$ with one from its complement, then we say that $A[S]$ is a $(1+\epsilon)$-local maximizer for the absolute determinant on the set of $r\times r$ non-singular principal submatrices of $A$.
\end{definition}

\begin{theorem}[\cite{XuFampaLee}]\label{thm:symapprox}
For a symmetric matrix $A\in\mathbb{R}^{n\times n}$, let $r:=\mathrm{rank}(A)$. Choose $\epsilon\ge0$, and let $\tilde{A}:=A[S]$ be a $(1+\epsilon)$-local maximizer for the absolute determinant on the set of $r\times r$ non-singular principal submatrices of $A$. The $n\times n$ matrix $H$ constructed by Theorem \ref{thm:symconstruction} over $\tilde{A}$, is a symmetric reflexive generalized inverse (having at most $r^2$ non-zeros), satisfying $\|H\|_1\le r^2(1+\epsilon)\|H_{ opt}^{sym,r}\|_1$, where $H_{ opt}^{sym,r}$ is a $1$-norm minimizing symmetric reflexive generalized inverse of $A$.
\end{theorem}


The idea of our algorithms in this section is then to select an $r\times r$  non-singular \emph{principal} submatrix $\tilde{A}$ of $A$, and construct the symmetric  reflexive generalized inverse with the inverse of this submatrix, as described in Theorem \ref{thm:symconstruction}.
The non-zero entries of $H$ will be the non-zero entries of $\tilde{A}^{-1}$. Guided by the result in Theorem \ref{thm:symapprox}, the `det' searches aim at selecting a principal submatrix $\tilde{A}$ that is a local maximizer
for the absolute determinant on the set of $r\times r$ non-singular principal submatrices of $A$.  We also apply `norm' searches, as done in the two previous sections.

In the following, we discuss how the  symmetric test matrices $A$ used in our computational experiments were generated, how we select the initial principal submatrix of $A$ to initialize the local searches, we describe the algorithms, and we present numerical results.

To analyze the local-search procedures, we  compare their solutions to the solution of LP problem  $P_{1}^{sym}$. Its solution corresponds to  $\| H_{opt}^{sym} \|_1$,  where $H_{opt}^{sym}$ is a $1$-norm minimizing symmetric generalized inverse of $A$.

\[
\begin{array}{lll}
(P_{1}^{sym}) \  z_{P_{1}^{sym}}:=&\min  & \left\langle J , T \right\rangle~,  \\
& \mbox{s.t.:}    &T - H \geq 0~,\\
		     &&T + H \geq 0~,\\
		     &&AHA=A~,\\
                  && H=H^{\top}~.
\end{array}
\]

We note that the result in Theorem \ref{thm:symapprox} is not related to $H_{opt}^{sym}$, but to $H_{opt}^{sym,r}$, a $1$-norm minimizing symmetric reflexive generalized inverse of $A$. However, from the proof  in \cite{XuFampaLee}, we can conclude  that the result is still valid if we replace $H_{opt}^{sym,r}$ by $H_{opt}^{sym}$ in the theorem. The reason of computing the second in our experiments is that, unlike the first, it can be efficiently  computed by the solution of an  LP problem, specifically $P_1^{sym}$.

\subsection{Our test matrices}
To test the local-search procedures proposed in this section, we  randomly generated 360 symmetric  matrices $A$ with varied dimensions,  ranks, and  densities.

The matrices were generated with the Matlab function \emph{sprandsym}.
 The function generates a random  $m\times m$ dimensional symmetric matrix $A$  with
  approximate density $d$ and eigenvalues
$rc$. The eigenvalues of $A$ are given as the input vector $rc$,.
The number of non-zero elements of $rc$
is of course the desired rank $r$.
For our experiments, we selected the $r$ nonzeros of $rc$ as before,
 $M\times(\pm\rho^{1},\pm\rho^{2},\ldots,\pm\rho^{r})$, where $M:=2$, and $\rho:=(1/M)^{(2/(r+1))}$, and
  the signs were randomly selected.

We divide our instances into the following three categories:
\begin{itemize}
\item Small: 90 instances. 5 with  each of the 18  combinations of the following parameters:  $m=n=50,80,100$; $r=0.1\times n,0.5\times n$;  $d = 0.25, 0.50, 1.00$.
\item Medium: 360 instances. 30 with  each of the 12 combinations of the following parameters:  $m=n=1000,2000$; $r=0.05\times n,0.1\times n$;  $d = 0.25, 0.50, 1.00$.
\end{itemize}

As the matrices are symmetric, we do not have the category of `Large' instances as in the previous  sections, where only $m$ was selected larger.

\subsection{Selecting an initial block for the local search}\label{secinitsym}

To construct an $r\times r$ non-singular principal submatrix of $A$ to initialize the local searches when $A$ is symmetric, we consider the following result.

\begin{proposition}
\label{propsym}
Let $A$ be a symmetric $m\times m$ matrix with rank $r$. Suppose that the $r$ columns of $A$ indexed by   $j_1,j_2,\ldots j_r$ are linear independent.  Then the principal submatrix $A[j_1,j_2,\ldots,j_r]$  has rank $r$.
\end{proposition}
\begin{proof}
Without loss of generality, we assume that $j_1=1,j_2=2,\ldots, j_r=r$, and
 \[
A =\left(\begin{array}{ll} \hat{A} &B\\ B^\top & D\end{array}\right),
\]
with $\hat{A}$ being an $r\times r$ symmetric submatrix. Then
\[
\mbox{rank}\left(\begin{array}{l} \hat{A} \\ B^\top \end{array}\right) = r.
\]
This implies that there exists an $r\times(m-r)$ matrix $X$, such that   $B = \hat{A} X$, $D = B^\top X$, as the $r$ first columns of $A$ form a basis for the column space of  $A$. Therefore,
\[
r=\mbox{rank}\left(\begin{array}{l} \hat{A} \\ B^\top \end{array}\right) = \mbox{rank}\left(\left(\begin{array}{c} I\\ X^\top \end{array}\right) \cdot \hat{A}\right) \leq \mbox{rank}(\hat{A}),
\]
 which implies that rank$(\hat{A})=r$.
 \qed
\end{proof}

Based on Proposition \ref{propsym}, we apply the same ideas described in  Algorithms \ref{AlgPhaseOnerow} and \ref{AlgGreedyrow}, but now to select only the set $S$ of $r$ linear independent rows of $A$. The set of columns $T$ is then selected to be equal to $S$.  

\subsection{The local-search procedures}

In Algorithm \ref{AlgLSFIsym}, we present the  `first improvement' local-search procedure  `FI(det)', which considers as the criterion for improvement of the given solution, the increase in the absolute determinant of the  $r\times r$ non-singular principal submatrix of $A$.
Based on Theorem \ref{thm:symapprox}, for a given rank-$r$ matrix $A$,  the procedure starts from a set $S$ of $r$ indices, such that $A[S]$ is non-singular.

\begin{algorithm}[!ht]
	\footnotesize{
		\KwIn{ $A\in \mathbb{R}^{m\times m}$, such that $r:=$rank($A$), $A=A^\top$\\$S\subset M:=\{1,\ldots,m\}$,   such that $|S|=r$, and $A[S]$ is non-singular. }
		\KwOut{ possibly updated set $S$. }		
$B:=A[S]$\;
$detB:=det(B)$\;
$\bar{M}:=M\setminus S$\;
$cont = true$\;
\While {($cont$)}
{
$cont=false$\;
\For {$\ell =1,\ldots, m-r$}
{
\For {$j =1,\ldots, r$}
{
$Saux:=S\cup\{\bar{M}(\ell)\}\setminus \{S(j)\}$\;
$B^+ :=A[Saux]$\;
$detB^+:=det(B^+)$\;
\If {$|detB^+| > |detB|$}
{
$B := B^+$\;
$detB := detB^+$\;
$S:=Saux$\;
$\bar{M}:=M\setminus S$\;
$cont=true$\;
break \;
}
}
}
}
\caption{`FI(det)'  for symmetric generalized inverses. \label{AlgLSFIsym} }
}
\end{algorithm}

In the  loop of Algorithm \ref{AlgLSFIsym} (lines 7--18), a column and row of $A[S, S]$ is replaced  if the absolute determinant increases with the replacement. To evaluate how much the determinant changes with the replacement, we consider  the result in Remark \ref{remcolsym}.

\begin{jremark}\label{remcolsym}
Let $\gamma\in \mathbb{R}^n$ and $A\in \mathbb{R}^{n\times n}$ with $\det(A)\neq 0$. Let $A_{\gamma/j}$ be the matrix obtained by replacing the $j^{th}$ column and row of
$A$ by $\gamma$ and $\gamma^\top$, respectively.
If $\hat{\alpha}\in \mathbb{R}^n$ solves the linear system of equations $A\alpha=\gamma$, then we have \[\det(A_{\gamma/j})= \hat{\alpha}_j^2 \times \det(A).\]
The result follows from
\[A_{\gamma/j} = [I +e_j(\hat{\alpha}-e_j)^\top] A [I+(\hat{\alpha}-e_j)e_j^\top].\]
\end{jremark}


The algorithm for  `FI$^+$(det)' differs from Algorithm \ref{AlgLSFIsym}, concerning the choice of the index  to be replaced in the current set $S$. For the `FI$^+$'  local search, instead of considering the first increase in the absolute determinant  of $A[S]$,   obtained in the loop described in lines 8--18, the algorithm computes the modification in the absolute determinant obtained for each index $j$, and selects the index that leads to the greatest increase. For the algorithm `BI(det)', the pair of indices $(\ell,j)$, in the two loops described in lines 7--18, that leads to the greatest increase in the absolute determinant is considered for the modification in the matrix.

For the `norm' searches, instead of computing the determinant of $B$ and $B^+$ in lines 2 and 11 of Algorithm \ref{AlgLSFIsym}, we compute the 1-norm of their inverses. The update in the index set $S$ occurs when the 1-norm decreases.


\subsection{Numerical results}
We initially consider the experiments done with the 90 instances in the `Small' category, which had the main purpose of analyzing the ratios between the
the 1-norm of the matrices $H$ computed by the three local searches based on the determinant, with the minimum 1-norm of a ah-symmetric generalized inverse given by the solution of the LP problem $P_{1}^{sym}$  $(\|H\|_1/z_{P_{1}^{sym}})$. We aim at checking how close these ratios are from the upper bound given by Theorem \ref{thm:symapprox}.

In   Figure \ref{barratplSymdens},  we present the average ratios   for the matrices with the same dimension, rank, and density.  From Theorem \ref{thm:symapprox}, we know that these ratios cannot be greater than $r^2$, and we see from the results, that for the matrices considered in our tests, we stay quite far from this upper bound. In general, the ratios increase with the rank $r$, the dimension $m=n$, and the density $d$  of the matrices,  but even for $r=50$, we  obtain ratios smaller than 2.1.

\begin{figure}[!ht]
\centering
\includegraphics[scale=0.5]{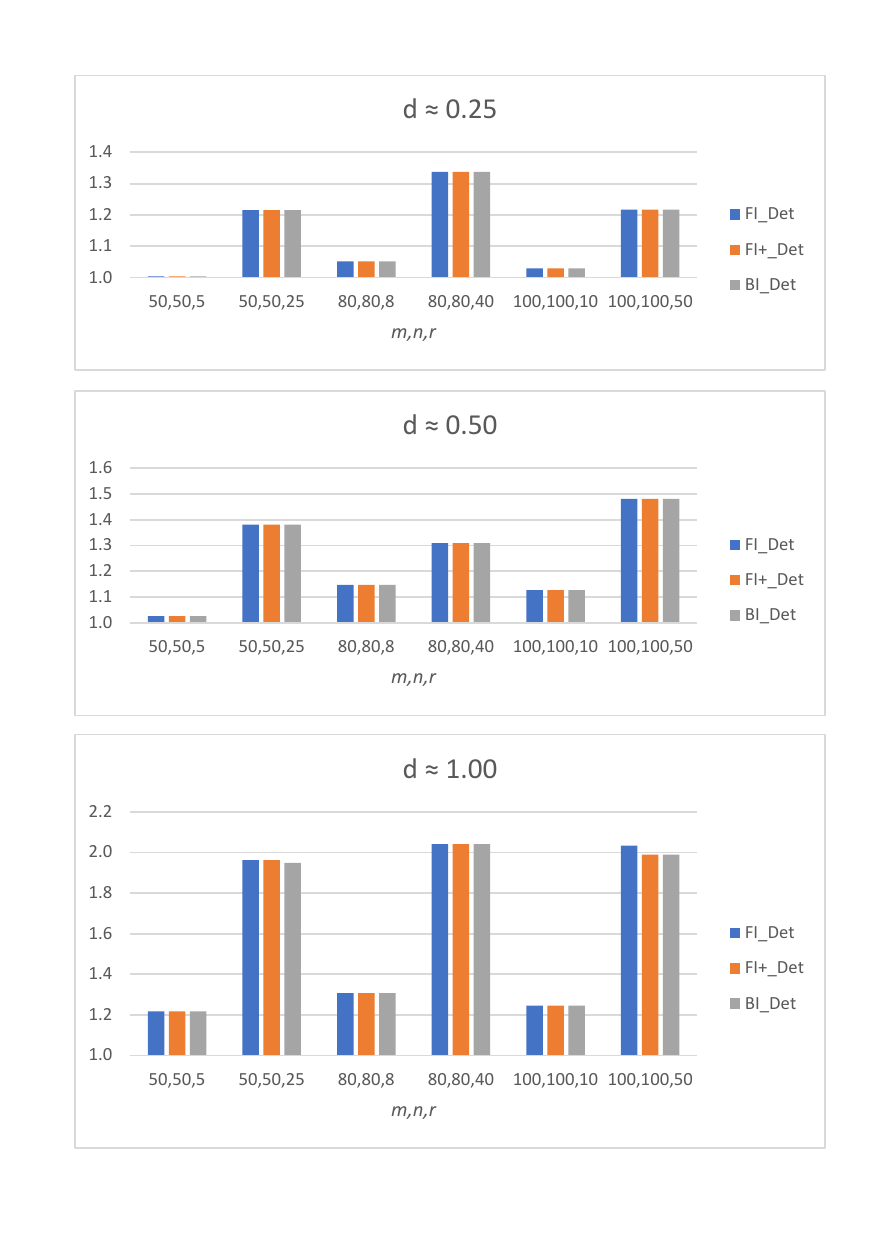}
\caption{ $\|H\|_1/z_{P_1^{sym}}$ (Small) (symmetric generalized inverse)}\label{barratplSymdens}
\end{figure}

In Table \ref{tab:1Sym}, besides presenting the average ratios depicted in Figure \ref{barratplSymdens}, we also  present the average running times to compute the generalized inverses. In case of the local searches, the total time to compute the generalized inverse is given by the sum of the time to generate the initial matrix $H$ by the simplified version of the NSub algorithm (Algorithm \ref{AlgNSub}) discussed in \S\ref{secinitsym},
and the time of the local search (FI(det), FI$^+$(det), or BI(det)).

\begin{table}[!ht]
	\centering
	\tiny
	\begin{tabular}{r|lll|rrlll}
		\hline
		&\multicolumn{3}{|c|}{ $\|H\|_1/z_{P^{sym}_{1}}$}&\multicolumn{5}{|c}{Time(sec)}\\
		\multicolumn{1}{c|}{$m,r,d$}& \multicolumn{1}{|c}{FI(det)}   & \multicolumn{1}{c}{FI$^+$(det)} & \multicolumn{1}{c|}{BI(det)} & \multicolumn{1}{|c}{$P^{sym}_{1}$} &   \multicolumn{1}{c}{NSub} &\multicolumn{1}{c}{FI(det)}   & \multicolumn{1}{c}{FI$^+$(det)} & \multicolumn{1}{c}{BI(det)} \\
		\hline
		50,05,0.25  & 1.003 & 1.003 & 1.003 & 0.623   & 0.031 & 0.004 & 0.004 & 0.003 \\
		50,25,0.25  & 1.215 & 1.215 & 1.215 & 6.238   & 0.004 & 0.028 & 0.025 & 0.028 \\
		80,08,0.25  & 1.051 & 1.051 & 1.051 & 2.041   & 0.029 & 0.007 & 0.004 & 0.004 \\
		80,40,0.25  & 1.338 & 1.338 & 1.338 & 46.916  & 0.027 & 0.134 & 0.115 & 0.140 \\
		100,10,0.25 & 1.029 & 1.029 & 1.029 & 4.132   & 0.020 & 0.008 & 0.007 & 0.009 \\
		100,50,0.25 & 1.218 & 1.218 & 1.218 & 69.780  & 0.038 & 0.222 & 0.196 & 0.174 \\
		50,05,0.50  & 1.027 & 1.027 & 1.027 & 0.641   & 0.003 & 0.003 & 0.003 & 0.003 \\
		50,25,0.50  & 1.380 & 1.380 & 1.380 & 7.511   & 0.003 & 0.034 & 0.028 & 0.054 \\
		80,08,0.50  & 1.148 & 1.148 & 1.148 & 2.785   & 0.005 & 0.005 & 0.004 & 0.004 \\
		80,40,0.50  & 1.309 & 1.309 & 1.309 & 46.378  & 0.027 & 0.125 & 0.095 & 0.108 \\
		100,10,0.50 & 1.127 & 1.127 & 1.127 & 4.704   & 0.005 & 0.010 & 0.008 & 0.010 \\
		100,50,0.50 & 1.480 & 1.480 & 1.480 & 97.951  & 0.022 & 0.256 & 0.235 & 0.186 \\
		50,05,1.00  & 1.218 & 1.218 & 1.218 & 0.908   & 0.003 & 0.003 & 0.002 & 0.003 \\
		50,25,1.00  & 1.963 & 1.963 & 1.950 & 10.777  & 0.003 & 0.041 & 0.034 & 0.048 \\
		80,08,1.00  & 1.307 & 1.307 & 1.307 & 3.628   & 0.004 & 0.006 & 0.004 & 0.005 \\
		80,40,1.00  & 2.042 & 2.042 & 2.042 & 62.370  & 0.008 & 0.200 & 0.170 & 0.386 \\
		100,10,1.00 & 1.246 & 1.246 & 1.246 & 6.318   & 0.006 & 0.012 & 0.010 & 0.015 \\
		100,50,1.00 & 2.034 & 1.991 & 1.991 & 942.775 & 0.016 & 0.315 & 0.298 & 0.398 \\
		\hline
	\end{tabular}
	\caption{Local Searches for symmetric generalized inverse vs. $P^{sym}_{1}$}\label{tab:1Sym}
\end{table}

We see from Figure \ref{barratplSymdens} and  Table \ref{tab:1Sym}, that  the three local searches converge to solutions of similar quality on most of the experiments. 
We also see in Table \ref{tab:1Sym} that the running times to solve $P_{1}^{sym}$ increase quickly with the dimension of the matrix, and are much higher than the times for the local searches.

Next, we consider the experiments done with the 360 instances in the `Medium' category, which had the main purpose of  comparing the different local searches  proposed.  
We  present in Table \ref{tab:2Sym} average results for each group of 30 instances with the same configuration, described in the first column. In the next three columns, we present statistics for the local searches based on the determinant, which are initialized  with the solutions given by the NSub  procedure, and in the last three columns we consider the application of the local searches based on the 1-norm of  $H$, which are initialized with  the solutions given by the three first local searches. In the first half of the table,  we show the mean and standard deviation of the 
 relative difference between  the 1-norm of the matrix $H$ obtained by each local search and the minimum value  among all of them, denoted by $||H_{best}||_1$. 

\begin{table}[!ht]
	\centering
	\tiny
\begin{tabular}{l|c|c|c|c|c|c}
		\hline
		\multicolumn{1}{c|}{$m,r,d$}& \multicolumn{1}{|c|}{FI(det)}   & \multicolumn{1}{|c|}{FI$^+$(det)} & \multicolumn{1}{|c|}{BI(det)} & \multicolumn{1}{|c|}{FI(det)}   & \multicolumn{1}{c}{FI$^+$(det)} & \multicolumn{1}{|c}{BI(det)} \\
		&   & & & \multicolumn{1}{|c|}{FI(norm)}   & \multicolumn{1}{c}{FI(norm)} & \multicolumn{1}{|c}{FI(norm)} \\
		\hline
		&\multicolumn{6}{|c}{ $(\|H\|_1 - \|H_{best}\|_1)/\|H_{best}\|_1$ }\vphantom{$\Sigma^{R}$} \\[3pt]
		\hline
		1000,050,0.25 & 0.040/0.037 & 0.040/0.037 & 0.040/0.037 & 0.000/0.000 & 0.000/0.000 & 0.000/0.000 \\
		1000,050,0.50 & 0.098/0.053 & 0.097/0.052 & 0.097/0.052 & 0.000/0.000 & 0.000/0.000 & 0.000/0.000 \\
		1000,050,1.00 & 0.073/0.038 & 0.073/0.038 & 0.073/0.038 & 0.001/0.004 & 0.001/0.004 & 0.000/0.000 \\
		1000,100,0.25 & 0.119/0.039 & 0.120/0.038 & 0.119/0.039 & 0.001/0.002 & 0.001/0.003 & 0.001/0.003 \\
		1000,100,0.50 & 0.077/0.049 & 0.077/0.049 & 0.077/0.049 & 0.000/0.000 & 0.000/0.000 & 0.000/0.000 \\
		1000,100,1.00 & 0.144/0.086 & 0.146/0.085 & 0.146/0.085 & 0.002/0.013 & 0.004/0.014 & 0.001/0.006 \\
		2000,100,0.25 & 0.107/0.049 & 0.107/0.049 & 0.107/0.049 & 0.000/0.000 & 0.000/0.001 & 0.000/0.000 \\
		2000,100,0.50 & 0.203/0.074 & 0.204/0.076 & 0.204/0.076 & 0.000/0.002 & 0.001/0.007 & 0.001/0.008 \\
		2000,100,1.00 & 0.187/0.126 & 0.187/0.126 & 0.187/0.126 & 0.000/0.000 & 0.000/0.000 & 0.000/0.000 \\
		2000,200,0.25 & 0.210/0.117 & 0.207/0.117 & 0.206/0.118 & 0.002/0.008 & 0.000/0.000 & 0.001/0.004 \\
		2000,200,0.50 & 0.218/0.094 & 0.216/0.093 & 0.218/0.094 & 0.001/0.003 & 0.001/0.007 & 0.003/0.009 \\
		2000,200,1.00 & 0.293/0.088 & 0.287/0.090 & 0.287/0.090 & 0.000/0.001 & 0.000/0.000 & 0.000/0.000 \\
		\hline
		&\multicolumn{6}{|c}{ Time(sec) }\\
		\hline
		1000,050,0.25 & 5.404/1.694     & 3.810/0.663    & 4.592/1.348     & 31.492/9.511      & 31.145/8.401      & 31.662/8.791      \\
		1000,050,0.50 & 23.957/6.956    & 17.517/4.616   & 26.754/12.052   & 266.19/71.67    & 268.51/71.74    & 268.53/70.63    \\
		1000,050,1.00 & 69.181/27.258   & 36.835/5.228   & 70.744/21.313   & 520.68/163.62   & 519.80/152.22   & 512.77/145.89   \\
		1000,100,0.25 & 577.53/190.48 & 310.02/44.63 & 734.92/207.46 & 5668.8/1073.3 & 5558.6/974.5  & 5559.5/1019.7 \\
		1000,100,0.50 & 5.320/2.249     & 3.372/0.593    & 4.191/1.303     & 31.881/9.784      & 30.835/8.241      & 31.580/9.405      \\
		1000,100,1.00 & 21.921/9.387    & 15.336/3.281   & 20.791/8.659    & 334.73/76.07    & 334.18/76.32    & 335.50/76.70    \\
		2000,100,0.25 & 66.531/18.351   & 37.912/5.941   & 73.078/27.199   & 667.30/196.24   & 658.59/197.53   & 665.20/199.51   \\
		2000,100,0.50 & 549.92/274.73 & 318.89/85.71 & 527.86/185.11 & 7116.8/1690.3 & 7010.6/1683.7 & 6844.4/1521.1 \\
		2000,100,1.00 & 4.421/2.456     & 3.251/1.033    & 3.923/1.855     & 45.350/10.200     & 45.420/10.154     & 45.121/11.686     \\
		2000,200,0.25 & 19.292/9.230    & 13.978/4.679   & 20.583/15.295   & 437.376/128.407   & 437.66/126.50   & 433.88/112.52   \\
		2000,200,0.50 & 48.824/21.987   & 35.423/9.395   & 41.201/18.144   & 921.35/280.99   & 935.07/310.43   & 935.30/297.93   \\
		2000,200,1.00 & 387.26/180.38 & 275.56/74.36 & 373.22/200.87 & 9525.2/2202.6 & 9675.2/2388.8 & 9645.7/2182.5 \\
		\hline
	\end{tabular}
	\caption{Local Searches for symmetric generalized inverse (Mean/Std Dev) (Medium) }\label{tab:2Sym}
\end{table}

Comparing to the tests with the `Small' instances,  we see that on this larger group of instances of higher dimension,  the similarity  among the quality of the solutions obtained by the three local searches based on the determinant is 
 still present. The average relative difference between the norms of the solutions obtained by these  searches and  the minimum norms    goes from  approximately 4 to 30\%, 
increasing with  the rank and the dimension. 
We also see that the  improvement on the solutions found by the local searches based on the 1-norm of $H$, when compared to the determinant searches comes with  a  high computational cost.  


In Table \ref{tab:3Sym}, we present the number of swaps for each local search.  Combining these results with the running time of the procedures, 
we conclude that  the FI$^+$(det)  procedure  is the best search based on the determinant, presenting smaller average computational times than the other two. We can also observe a relative increase in the number of swaps when compared to  the non-symmetric cases discussed in the previous sections. This increase is also reflected in the greater improvement obtained with these searches. However, the improvement continues to come with a very  high computational cost.

\begin{table}[!ht]
	\centering
	\tiny
\begin{tabular}{r|rrrrrr}
		\hline
		\multicolumn{1}{c|}{$m,r,d$}& \multicolumn{1}{|c}{FI(det)}   & \multicolumn{1}{c}{FI$^+$(det)} & \multicolumn{1}{c}{BI(det)} & \multicolumn{1}{c}{FI(det)}   & \multicolumn{1}{c}{FI$^+$(det)} & \multicolumn{1}{c}{BI(det)} \\
		&   & &  & \multicolumn{1}{c}{FI(norm)}   & \multicolumn{1}{c}{FI(norm)} & \multicolumn{1}{c}{FI(norm)} \\
		\hline
		1000,050,0.25 & 3.633 & 2.067 & 1.633 & 4.467  & 4.467  & 4.533  \\
		1000,050,0.50 & 4.833 & 3.267 & 2.267 & 12.800 & 12.700 & 12.700 \\
		1000,050,1.00 & 7.133 & 3.767 & 3.000 & 11.667 & 11.667 & 11.633 \\
		1000,100,0.25 & 9.233 & 5.600 & 3.967 & 28.733 & 28.833 & 28.867 \\
		1000,100,0.50 & 3.533 & 1.733 & 1.500 & 6.200  & 6.200  & 6.200  \\
		1000,100,1.00 & 3.367 & 2.000 & 1.567 & 16.967 & 16.933 & 16.900 \\
		2000,100,0.25 & 7.167 & 3.900 & 3.100 & 15.533 & 15.533 & 15.533 \\
		2000,100,0.50 & 6.433 & 3.433 & 2.500 & 40.733 & 40.467 & 40.267 \\
		2000,100,1.00 & 3.100 & 1.867 & 1.333 & 12.767 & 12.767 & 12.767 \\
		2000,200,0.25 & 3.333 & 1.967 & 1.533 & 25.333 & 25.333 & 25.033 \\
		2000,200,0.50 & 3.000 & 1.833 & 1.333 & 30.000 & 29.900 & 29.867 \\
		2000,200,1.00 & 3.200 & 1.933 & 1.467 & 60.700 & 60.067 & 59.800 \\
		\hline
	\end{tabular}
	\caption{Number of swaps (Medium) (symmetric generalized inverse) }\label{tab:3Sym}
\end{table}

\section{Case Study}\label{seccase}

In this section, we report on a case study that we 
undertook on a real-world data set. We sought to validate our techniques,
in the ah-symmetric case. We applied the ideas as we described,
but additionally in a somewhat more general way.
In our presentation, we always worked with $r$ equal to the
rank of the input data matrix. But we can also take \emph{smaller} $r$, with the benefit of gaining even sparser (block) ah-symmetric
generalized inverses. Of course, with smaller $r$, we give up something in the least-squares fit, and so we explored this trade off in our case study.

We applied our techniques to the ``Communities and Crime Data Set'' (\url{https://archive.ics.uci.edu/ml/datasets/Communities+and+Crime})
obtained  from
the UCI Machine Learning Repository, at  the Center for Machine Learning and Intelligent Systems, University of
California at Irvine. The data set combines socioeconomic data from the 1990 US Census, law enforcement data from the 1990 US LEMAS (Law Enforcement Management and Administrative Statistics) survey, 
 and crime data from the 1995 FBI UCR (Uniform Crime Reporting) Program.
 
 The data is for $1,994$ communities, $128$ variables: 122 predictive, 5 non-predictive, 1 goal.
 The goal variable is the total number of violent crimes per 100,000 population.
 Data was incomplete for one community and for 22 predictive variables.
 So we settled on a (rather dense) $(m=1,993)\times (n=100)$ matrix $A$ to work with,
 and a corresponding goal $b\in\mathbb{R}^m$. 
 Considering the real singular value decomposition $A=\sum_{i=1}^n \sigma_i u_i v_i^\top$  of $A$,
 with $\sigma_1\geq \sigma_2\geq \cdots \geq \sigma_n \geq 0$, orthonormal $u_i\in\mathbb{R}^m$, and orthonormal $v_i\in\mathbb{R}^n$,
 we define low-rank versions of $A$, namely $A_r:= \sum_{i=1}^{r} \sigma_i u_i v_i^\top$,
 for $r=50,49,\ldots,10$. 
 In Figure \ref{fig_casestudySV}, we plot the singular values of $A$.
 We note that $A_{r}$ is
 the closest rank-$r$ matrix in Frobenius norm to $A$. For $r=50$, we have 
 $\|A_{50}\|_F^2=40,353$ as compared to  $\|A\|_F^2=40,480$, so we can say that 
 $A_{50}$ is a very close approximation of $A$. 
 But considering the sharp decay in the plot, we can even say that 
 $A_{r}$ is a good approximation of $A$ for all $r=50,49,\ldots,10$.


 \begin{figure}[!ht]
\centering
\includegraphics[width=0.95\textwidth]{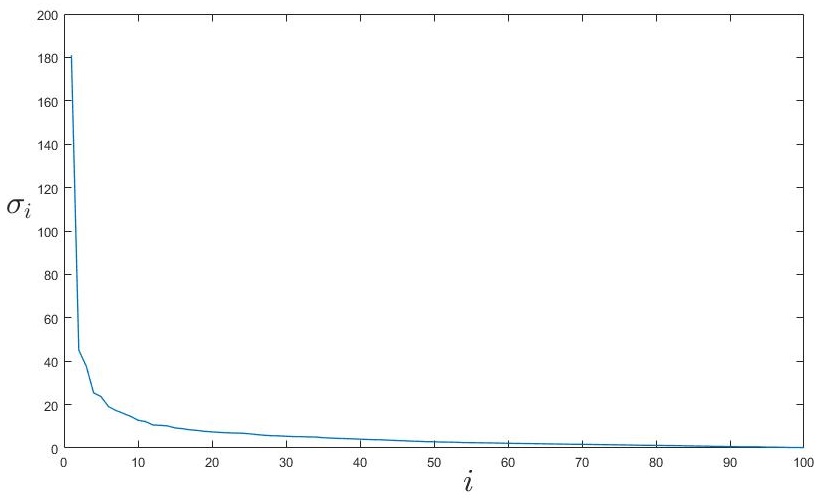}
\caption{Singular values of $A$}\label{fig_casestudySV}
\end{figure}


\noindent
In Figures \ref{fig_casestudy12}--\ref{fig_casestudy34}, we present the  R-squared\footnote{percentage of variation of the goal variable that is linearly explained
by the regression} statistical measure computed  by considering:
\begin{itemize}
    \item all 100 columns of the matrix $A_{50}$: indicated by the gold line in all four figures.
    The R-squared for $A_{50}$ is 0.6690, only a bit lower than the $R$-squared of 0.6957 for $A$;
    \item the $r$ columns of $A_{50}$ selected by by applying our local search to $A_{50}$, for $r=50,49,\ldots,10$: indicated by the blue plot in Figures \ref{fig_casestudy12} and \ref{fig_casestudy13};
       \item the $r$ columns of $A_{50}$ selected by applying our local search to $A_r$, for $r=50,49,\ldots,10$, indicated by the orange plot in Figures \ref{fig_casestudy13} and \ref{fig_casestudy34};
      \item the $r$ columns of $A$ selected by by applying our local search to $A_{50}$, for $r=50,49,\ldots,10$: indicated by the black plot in Figures \ref{fig_casestudy12} and \ref{fig_casestudy24};
       \item the $r$ columns of $A$ selected by applying our local search to $A_r$, for $r=50,49,\ldots,10$, indicated by the red plot in Figures \ref{fig_casestudy24} and \ref{fig_casestudy34};     
\end{itemize}

The local search procedure applied in these experiments was FI$^+$(det), which had the best performance on the tests presented in \S \ref{results_ahrg}.

We wish to emphasize that our local searches do not consider the goal variable.
Rather, our local searches aim to find a good \emph{small} set of columns of $A$ (corresponding to a sparse block-structured ah-symmetric generalized inverse), 
that can be good for \emph{any} realizations of the goal variable.

Figure \ref{fig_casestudy12} considers our local search for finding $r$ columns, 
 for $r=50,49,\ldots,10$, always applying the
local search to $A_{50}$.  In the end, we compare the performance of the selected 
columns indices, doing the regressions on the chosen columns of both $A_{50}$ and $A$. 
We can see that it does not matter much whether we do the regressions on $A_{50}$ or $A$; that is,
$A_{50}$ is a reasonable low-rank approximation of $A$. Also, we 
see only a slow deterioration in R-squared as we decrease $r$ (using either  $A_{50}$ or $A$);
so our algorithms do well even for $r=10$. 

Figure \ref{fig_casestudy13} 
 compares two different local searches for finding $r$ columns, for $r=50,49,\ldots,10$.
One always does the local search on $A_{50}$, while the other does it on $A_r$. In the end,
we evaluate the local searches by doing  regressions on the chosen columns of $A_{50}$.
We can see that both local searches
perform similarly (possibly the one using $A_r$ is a bit better), with slow deterioration in R-squared as we decrease $r$. 

Figure \ref{fig_casestudy24} again compares the two local searches, 
but we evaluate the local searches by doing  regressions on the chosen columns of $A$ (rather than $A_{50}$.
We reach the same conclusion as we did for Figure \ref{fig_casestudy13}. 

Finally,   Figure \ref{fig_casestudy34}, 
considers our local search for for finding $r$ columns, 
 for $r=50,49,\ldots,10$, always applying the
local search to $A_{r}$. In the end, we compare the performance of the selected 
columns indices, doing the regressions on both the chosen columns of $A_{50}$ but also on $A$.
We reach the same conclusion as we did for Figure \ref{fig_casestudy12}.

Overall, we find that for our local-search algorithm for finding $r=50$, $49,\ldots,10$ good regression variables: 
(i) it is quite robust to versions of the input matrix,
working well on $A_{50}$ or $A_r$, (ii) it is quite robust to how we evaluate the chosen $r$ column indices
(treating either $A$ or its low-rank counterpart $A_{50}$ as ``the truth''), and (iii)
we get very little deterioration in the quality of the least-square fits (as measured by R-squared),
as we decrease $r$.



\begin{figure}[!ht]
\centering
\includegraphics[width=1.00\textwidth]{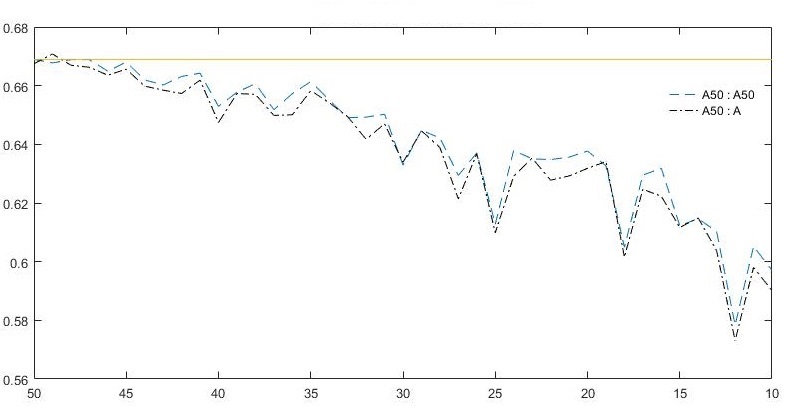}
\caption{Local search on $A_{50}$, Regression on $A_{50}/A$}\label{fig_casestudy12}
\end{figure}

\begin{figure}[!ht]
\centering
\includegraphics[width=1.00\textwidth]{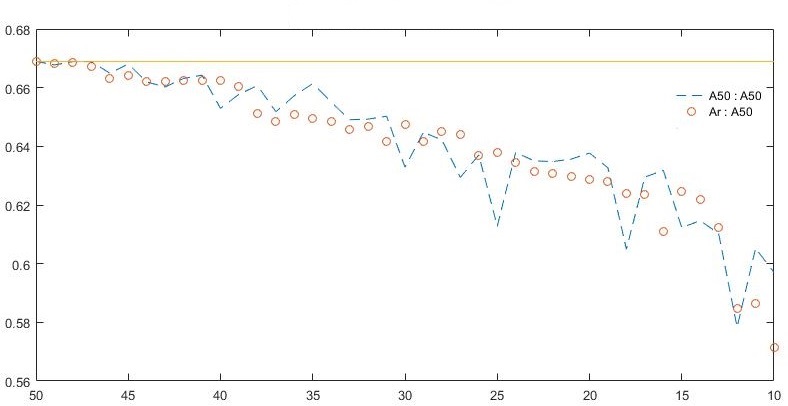}
\caption{Local search on $A_{50}/A_r$, Regression on $A_{50}$}\label{fig_casestudy13}
\end{figure}

\FloatBarrier

\begin{figure}[!ht]
\centering
\includegraphics[width=1.00\textwidth]{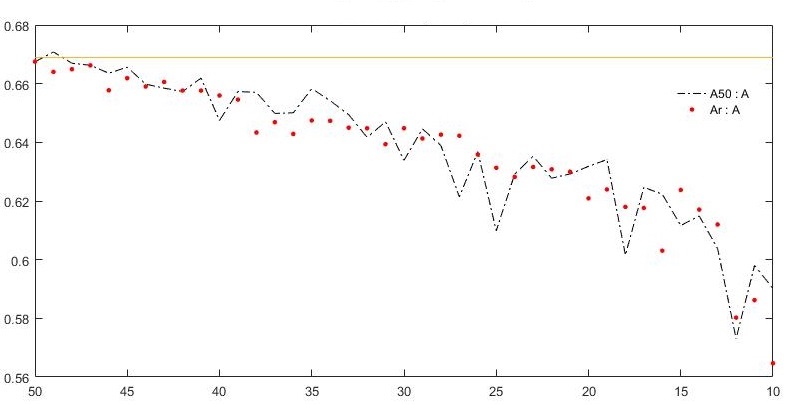}
\caption{Local search on $A_{50}/A_r$, Regression on $A$}\label{fig_casestudy24}
\end{figure}

\begin{figure}[!ht]
\centering
\includegraphics[width=1.00\textwidth]{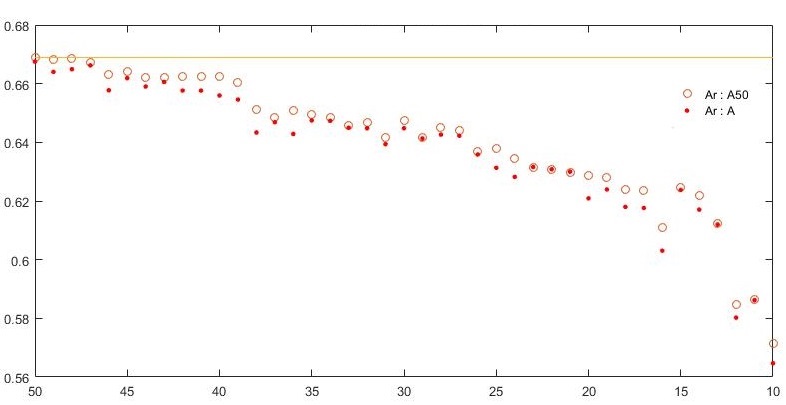}
\caption{Local search on $A_{r}$, Regression on $A_{50}/A$}\label{fig_casestudy34}
\end{figure}



Finally, we looked a bit more carefully at the attributes selected by the local searches
for $r=20$, chosen as giving a good level of prediction for a rather low rank. 
Interesting, there is only agreement on twelve of the twenty attributes selected by the two local searches, 
while the models have 
very similar predictive capability. 
We can see that drawing causal conclusions from the selected attributes would be very dubious.

\section{Concluding remarks}\label{sec:conc}

  We have demonstrated that the local-search procedures presented in \cite{FampaLee2018ORL,XuFampaLee}  can be successfully implemented to construct sparse,  block-structured reflexive generalized inverses with different properties.
  We find that the performance (1-norm achieved) is much better than tight worst-case guarantees.
   Overall, we find that the search procedures are very robust in terms of
   many of the algorithmic choices that need to be made. For scaling purposes,
   we found that it is necessary to be mindful of the numerics and of economizing when
   seeking local improvements, and calculating initial solutions efficiently
   proves to be a surprisingly difficult practical issue.
%

 \cite{FampaLee2018ORL,XuFampaLee} established that the ratios between the norms of the  solutions of the local searches and the LP problems $P_1$, $P_{123}$, and $P_1^{sym}$ are bounded by $r^2$, $r$, and $r^2$, respectively, when considering generalized inverses, ah-symmetric generalized inverses, and symmetric generalized inverses.
We observed in our numerical  experiments  that the average ratios  were much smaller than these worst-case upper bounds, and also that they were smaller for the ah-symmetric case.  This can be explained by the fact that the upper bound is smaller for the ah-symmetric generalized inverses
($r$ vs. $r^2$), but also because in this case, we could include the linearized constraints for property \ref{property2} in the LP problem $P_{123}$, increasing its optimal objective function value.


 Comparing the three local-search procedures based on the determinant, we conclude that they converge to solutions of very similar quality.   The best improvement approach (`BI(det)') is too expensive and can be discarded. In general, the procedure `FI$^{+}$(det)'  had slightly better times than `FI(det)'.

 The computational time to solve the LP problems considered is much larger than the times of the local searches, and increases much faster than the times of the local searches when the dimension, the rank, or the density of the matrices increases. So we conclude that LP is not a competitive alternative to the local searches, even if we only cared about running time.
     An interesting point is that the most costly LP solution is given for problem $P_{123}$, with more constraints to model ah-symmetric generalized inverses. On the other side, the local-search procedures to construct these matrices are the fastest ones, as the searches are only applied to the columns of the matrices, for a given set of linear independent rows.

 The local-search procedures based on the 1-norm were considered with the purpose of determining
 whether or not  the searches based on the determinant could be still  improved with respect to the 1-norm of the matrices. For the ah-symmetric case, we saw only a relatively modest improvement in
 1-norm. For generalized inverses we saw better improvements in 1-norm, and for 
 symmetric generalized inverses even better. We can conclude that for the ah-symmetric case,
 1-norm search is never recommended, and for the others, perhaps they could be considered if
 one is willing to incur a substantial computational cost. 

 The running times of the local-search procedures based on the determinants were critically  decreased with the use of the results pointed out in our remarks (Remarks \ref{remcol}, \ref{remrow}, \ref{remnormcol}, \ref{remnormcolpinv}, \ref{remcolsym}) which indicate how to efficiently update the determinant of the matrices after the rows and columns swaps at each iteration. A na{\"\i}ve implementation, instead recomputing determinants from scratch, would not allow to scale to large instances.

A significant part of our effort spent in this research was dedicated to developing a good algorithm to construct an initial solution to our local searches. The computation of an $r\times r$ non-singular submatrix of a rank-$r$ matrix turned out to be a challenge when considering our large, and even medium-sized test instances. The procedure proposed had a very good performance in our numerical experiments and  its Matlab  implementation is now available through Mathworks.  

\section*{Acknowledgments}
M. Fampa was supported in part by CNPq grant 303898/2016-0.
J. Lee was supported in part by ONR grant N00014-17-1-2296.
G. Ponte was supported in part by CNPq PIBIC scholarship 149149/2020-4.  
M. Fampa, J. Lee and L. Xu were supported in part by funding from the Simons
Foundation and the Centre de Recherches Math\'ematiques, through the Simons-CRM
scholar-in-residence program.



\bibliographystyle{plain}
\bibliography{ginv}

\end{document}